\documentclass{amsart}
\usepackage[utf8]{inputenc}
\usepackage[T1]{fontenc}
\usepackage[cal=dutchcal]{mathalfa}
\usepackage{amsmath,amssymb,enumerate}
\usepackage{epsf,epsfig,amsfonts,graphicx,color,url}
\usepackage{mathrsfs}
\usepackage{subcaption}
 \usepackage{svg}
\usepackage[normalem]{ulem}

\usepackage{tikz-cd}

\usepackage[autostyle]{csquotes}

\usepackage{filecontents}

\DeclareMathOperator{\Com}{Com}

\DeclareMathOperator{\V}{V}

\DeclareMathOperator{\Int}{Int}

\def\T{\mathcal{T}}
\def\R{\mathbb{R}}

\def\C{\mathcal{C}}

\def\D{\mathcal{D}}

\def\Ri{Riemannian }

\numberwithin{equation}{section}

\newtheorem{theorem}{Theorem}

\newtheorem{thm}{Theorem}
\newtheorem{proposition}[thm]{Proposition}

\newtheorem{lemma}[thm]{Lemma}
\newtheorem{definition}[thm]{Definition}
\newtheorem{conj}[thm]{Conjecture}

\newtheorem{remark}[thm]{Remark}

\renewcommand{\emph}[1]{{\bfseries\itshape{#1}}}

\usepackage{hyperref}
\usepackage[%
	capitalize,nameinlink
]{cleveref}

\hypersetup{%
	colorlinks=true,
	linkcolor=blue,
	citecolor=magenta,
	urlcolor=magenta,
}

\numberwithin{figure}{section}

\def\Ri{Riemannian }
\def\SR{Sub-Riemannian }

\def\sR{sub-Riemannian } 

\def\ma{metabelian }

\newcommand{\Lag}{\mathfrak{g}}
\newcommand{\Laa}{\mathfrak{a}}

\newcommand{\Ag}{\mathbb{A}}
\newcommand{\G}{\mathbb{G}}

\newcommand{\PJ}{\mathcal{J}^2(\mathbb{R},\mathbb{R}^2)}
\newcommand{\kJ}{\mathcal{J}^k(\mathbb{R},\mathbb{R}^n)}

\begin{document}

\title[Metric Lines in the Space of Curves.]{Metric Lines in the Space of Curves}
\author[D. Catalá]{Daniella\ Catalá} 
\address{Daniella Catalá: Department of Mathematics, Massachusetts Institute of Technology, Cambridge, MA 02139, U.S.}
\email{dcatala@mit.edu}
\author[M. Vollmayr-Lee]{Miriam\ Vollmayr-Lee} 
\address{Miriam Vollmayr-Lee: Department of Mathematics, Grinnell College, Grinnell, IA 50112, U.S.}
\email{vollmayr@grinnell.edu}  
\author[A.\ Bravo-Doddoli]{Alejandro\ Bravo-Doddoli} 
\address{Alejandro Bravo-Doddoli: Department of Mathematics, University of Michigan, Ann Arbor, MI 48109, U.S.}
\email{Abravodo@umich.edu}

\begin{abstract} 
This paper investigates sub-Riemannian geodesics in the jet space of curves. We establish the existence of two distinct families of metric lines, that is, globally minimizing geodesics, in the $2$-jet space of plane curves. 
This result contributes to the broader classification of metric lines in jet spaces. Additionally, we present precise criteria that characterize when a sub-Riemannian geodesic in the $2$-jet space of plane curves is a metric line.
\end{abstract}

\maketitle

\section*{2020 Mathematics Subject Classification}

53C17, 53C22, 49K15, 20F65, 22E25.

\section{Introduction}

In this work, we employ the sequence method to prove the existence of two distinct families of metric lines in the $2$-jet space of plane curves, denoted $\PJ$. This result provides an initial step toward resolving the conjecture on the classification of metric lines in $\PJ$. More broadly, the $k$-jet space of curves $\kJ$ admits the structure of a Carnot group \cite{warhurstjet} and therefore carries a natural sub-Riemannian structure. To place our study in context, consider a geodesic on a sub-Riemannian manifold $M$, meaning a curve that locally minimizes arc length. A central question then arises: under what conditions is a sub-Riemannian geodesic globally minimizing? We define a metric line as a geodesic $\gamma:\R\to M$ that is globally minimizing.

The sequence method, introduced by the third author in  \cite{BravoDoddoli+2024,bravododdoligeltype}, offers a systematic framework for determining whether a given candidate geodesic is a metric line. 
The process begins by classifying geodesics in $\kJ$ into line, periodic, homoclinic, turn-back, and direct-type geodesics; see Section \ref{sub-sec:class-geo}. We then select a candidate geodesic based on its asymptotic velocity. We now state our conjecture on metric lines in $\kJ$.

\begin{conj}\label{con:jet-space}
    The metric lines in $\kJ$ are precisely the line, homoclinic, and direct-type geodesics. 
\end{conj}
Every Carnot group admits a family of sub-Riemannian geodesics of line type, all of which are metric lines. In contrast, periodic and turn-back geodesics in $\kJ$ are not metric lines \cite{bravododdoligeltype}. 

The first main contribution is as follows.
\begin{theorem}\label{the:main}
    In addition to line geodesics, $\PJ$ admits two families of metric lines of homoclinic type. Modulo Carnot translations and dilations, these are one-parameter families defined in Section \ref{subsubsec:cand-geo}. 
\end{theorem}
 We make two observations concerning Theorem \ref{the:main}. First, $\PJ$ is a 3-step Carnot group; then it does not admit direct-type geodesics; see Proposition \ref{prp:only-hom}. Second, modulo Carnot translations and dilations, the full family of sub-Riemannian geodesics of homoclinic type in $\PJ$ is two-parameter; see Section \ref{subsec:two-jet}. The two one-parameter families identified in Theorem \ref{the:main}
 are special subfamilies of this general family. 
 

 The second main result is Theorem \ref{thm:main-2}, which gives a criterion for determining when a \sR geodesic in $\PJ$ is a metric line. We prove Theorem \ref{thm:main-2} using the sequence method. A broader goal of the sequence method is to address Conjecture \ref{conj} concerning the classification of metric lines in Carnot groups. This conjecture states that metric lines are precisely the geodesics exhibiting a particular asymptotic velocity. Proposition \ref{prop:met-char} gives a necessary condition for a \sR geodesic to be a metric line. This condition relates the geodesic’s average velocity to its asymptotic velocity. In particular, Proposition \ref{prop:met-char} shows that periodic and turn-back geodesics are not metric lines.

Another contribution of this paper is to clarify the role of the asymptotic velocity introduced above. Given a candidate geodesic, we construct an auxiliary sub-Riemannian manifold $\R^5_{(a,b)}$. The sequence method relies on the period map, which encodes the asymptotic behavior of the sub-Riemannian geodesics.  
A crucial requirement of the sequence method is that the period map be one-to-one. In previous work \cite{BravoDoddoli+2024,bravododdoligeltype,bravo2022geodesics}, we considered Carnot groups for which the period map had a one-dimensional domain, which enabled us to establish this requirement using tools from calculus. In the case of $\R^5_{(a,b)}$, the period map has a two-dimensional domain. In this paper, we use the Hadamard global diffeomorphism theorem to show that the period map is one-to-one for the two families considered in Theorem \ref{the:main} \cite{krantz2002implicit}.

The proof of Conjecture \ref{con:jet-space} for $n=2$ and $k=2$ remains open for the general two-parameter family of homoclinic geodesics. In this paper, we establish a partial result for the two one-parameter families corresponding to \(b=0\), \(a\neq0\), and \(a=0\), \(b\neq0\). Extending this result to the case \(ab\neq0\) remains open.

\section*{Structure of the paper}
Section \ref{sec:pre} introduces the foundational concepts and tools used throughout. Section \ref{subsec:sR-man} provides a concise overview of sub-Riemannian geometry and constructs the sub-Riemannian structure on $\kJ$. Section \ref{subsec:two-jet} focuses on the specific case of $\PJ$, where we parametrize the family of homoclinic geodesics relevant to our study in terms of the parameters $(a,b)$. Section \ref{subsec:mag-space} introduces the auxiliary sub-Riemannian space $\R^5_{(a,b)}$, which serves as the framework for applying the sequence method. Section \ref{sec:sec-method} presents the sequence method, states and proves  Theorem \ref{thm:main-2}, and concludes by proving Theorem \ref{the:main} as a direct consequence. Appendix \ref{ap:lie-alg} presents the Carnot structure of $\kJ$ along with its symmetries. Appendix \ref{Ap:abn} contains the proofs of the propositions characterizing abnormal geodesics in both $\kJ$ and $\R^5_{(a,b)}$.

\section{Preliminaries}\label{sec:pre}

This section introduces the fundamental concepts and constructions needed to prove our main results. Section \ref{subsec:sR-man} presents  $\kJ$ as a sub-Riemannian manifold. Section \ref{subsubsec:Carnot} reviews Carnot groups, presents Conjecture \ref{conj} on the classification of metric lines, and states Proposition \ref{prop:met-char}, which gives a necessary condition for a geodesic to be a metric line. Section \ref{subsec:jet-spa-car} develops the sub-Riemannian structure of the jet space $\kJ$; Section \ref{sss:geo-con} outlines the construction of its geodesics, and Section \ref{sub-sec:class-geo} provides their classification. Section \ref{subsec:two-jet} presents the essential properties of $\PJ$ relevant to the proof of Theorem \ref{the:main}. Section \ref{subsec:mag-space} describes the sub-Riemannian structure of $\R^5_{(a,b)}$. Section \ref{sss:mag-spa} constructs $\R^5_{(a,b)}$,  Section \ref{sss:geo-mag} describes the sub-Riemannian geodesics in $\R^5_{(a,b)}$, and Section \ref{sss:hom-goe} parametrizes the set of homoclinic geodesics. Section \ref{sss:cost-fuc} introduces the cost function, an auxiliary tool used to define the period map, and Section \ref{sss:per-map} defines the period map and examines its key properties. Finally, Section \ref{sss:sec-geo} presents essential results on sequences of geodesics.

\subsection{Jet Space as a \SR Manifold}\label{subsec:sR-man}

A \sR manifold is a triple $(M,\D,\langle\cdot,\cdot\rangle)$, where $M$ is a smooth manifold, $\D$ is a bracket-generating distribution, and $\langle\cdot,\cdot\rangle$ is an inner product on $\D$ \cite{montgomery2002tour,agrachev2019comprehensive}. A curve $\gamma(t)$ is called horizontal if $\dot{\gamma}(t) \in \D_{\gamma(t)}$ whenever  $\dot{\gamma}$ is defined. Chow's theorem states that if a distribution $\D$ is bracket-generating, then any two points in $M$ are connected by a horizontal curve \cite{montgomery2002tour}. The \sR arc length of a horizontal curve is defined in the usual way, as in \Ri geometry \cite{agrachev2019comprehensive,montgomery2002tour}. A curve is called a geodesic if it is locally minimizing. The Pontryagin maximum principle gives necessary conditions for a geodesic to be locally minimizing and distinguishes normal from abnormal geodesics. The \sR kinetic energy is a Hamiltonian function $H_{sR}:T^*M \to \R$ whose Hamiltonian trajectories project to normal geodesics in $M$. Abnormal geodesics are specific to sub-Riemannian geometry and have no Riemannian counterpart. The primary goal of this paper is to study the classification of metric lines in $\PJ$.

\begin{definition}
    Let $M$ be a \sR manifold, and let $dist_M(\cdot,\cdot)$ be the \sR distance on $M$. A curve $\gamma:\R \to M$ is a metric line if it is a globally minimizing geodesic; that is, 
    $$ |a-b| = dist_{M}(\gamma(a),\gamma(b)) \;\text{for all}\;a,b\in\R. $$
\end{definition}
Alternative terms for \enquote{metric line} are \enquote{globally minimizing geodesic,} \enquote{isometric embedding of the real line,} and \enquote{infinite geodesic} \cite{agrachev2019comprehensive,hakavuori2023blowups}. 

The following concept is fundamental to this work.
\begin{definition}
    Let $(M,\D_M,\langle\cdot,\cdot\rangle_M)$ and $(N,\D_N,\langle\cdot,\cdot\rangle_N)$ be \sR manifolds, and let $\phi:M \to N$ be a submersion. The submersion $\phi$ is called \sR if it respects the \sR structures, i.e., $\phi_* \D_M = \D_N$ and $\phi^*\langle\cdot,\cdot\rangle_N = \langle\cdot,\cdot\rangle_M$ where $\phi_*$ and $\phi^*$ are the push-forward and pull-back of $\phi$, respectively.  
\end{definition}
Sub-Riemannian submersions are a special case of submetries in metric geometry; see \cite[Definition 3.1.23]{le2025metric}. Every sub-Riemannian submersion admits horizontal lifts of horizontal curves; see \cite[Proposition 3.1.24]{le2025metric}. A classical result in sub-Riemannian geometry is the following.
\begin{lemma}[{\cite[Proposition 1]{bravo2022geodesics}}]\label{lemm:hor-lift-met-lin}
 Let $\phi:M\to N$ be a \sR submersion (or submetry). If $\gamma(t)$ is a horizontal lift of a metric line in $N$, then $\gamma(t)$ is a metric line in $M$. 
\end{lemma}

\subsubsection{ Carnot Groups and the General Conjecture}\label{subsubsec:Carnot}

A Carnot group $\G$ is a simply connected Lie group whose Lie algebra $\Lag$ is stratified and nilpotent, i.e., the Lie algebra satisfies
$$ \Lag = \V_1 \oplus \dots \oplus \V_s, \;\;\; [\V_1,\V_{i}] = \V_{i+1},\;i=1,\dots,s-1,\;\text{with} \;\;\; [\V_1,\V_{s}] = \{0\}, $$
where $[\cdot,\cdot]$ is the Lie bracket and the integer $s$ is called the step of $\G$ \cite{montgomery2002tour,agrachev2019comprehensive,le2025metric}. 

Every Carnot group $\G$ admits a canonical projection $\pi:\G\to \V :=  \G / [\G,\G]$, where  $\V\simeq \V_1$. In addition, we equip $\G$ with a left-invariant distribution $\D:= (L_g)_*V_1$, where $L_g$ is the left translation by the element $g\in \G$, and $(L_g)_*$ is the push-forward of $L_g$.  We equip \(V_1\simeq V\) with a Euclidean inner product and extend it left-invariantly to \(\mathcal D\), thereby obtaining a left-invariant sub-Riemannian metric on $\G$ \cite{agrachev2019comprehensive}.  With this structure, $\pi$ is a \sR submersion. Therefore, if $\gamma(t)$ is the horizontal lift of a straight line in $\V$, then Lemma \ref{lemm:hor-lift-met-lin} implies that $\gamma(t)$ is a metric line in $\G$. These sub-Riemannian geodesics are called line geodesics; see Section \ref{sub-sec:class-geo} for more details.

With the above discussion in place, a natural question arises: beyond line geodesics, which geodesics are metric lines in a Carnot group $\G$? The following conjecture, proposed by the third author, addresses this question.
\begin{conj}\label{conj}
    Consider a left-invariant \sR structure on a Carnot group $\G$. The metric lines in $\G$ are precisely the \sR geodesics $\gamma(t)$ that are parametrized by arc length and  for which there exists a unit vector $v\in\V_1$ such that
\begin{equation*}
   v =  \lim_{t \to - \infty} (L_{\gamma^{-1}(t)})_* \dot{\gamma}(t) = \lim_{t \to  \infty} (L_{\gamma^{-1}(t)})_* \dot{\gamma}(t).
\end{equation*}
\end{conj}

The following result gives a necessary condition for a geodesic to be a metric line. 
\begin{proposition}[{\cite[Proposition 2.6]{bravododdoligeltype}}]\label{prop:met-char}
 Under the hypotheses of Conjecture \ref{conj}, a necessary condition for $\gamma(t)$ to be a metric line is that there exists a unit vector $v\in V_1$ such that
        \begin{equation}\label{eq:in-met-char}
            1 = \limsup_{t \to \infty} \frac{1}{2t} \int_{-t}^t <v,(L_{\gamma^{-1}(s)})_* \dot{\gamma}(s)> ds . 
        \end{equation} 
\end{proposition}
The condition in Conjecture \ref{conj} implies the condition in Proposition \ref{prop:met-char}.

\subsubsection{The Jet Space}\label{subsec:jet-spa-car}

We construct $\mathcal{J}^k(\R,\R^n)$ and present its sub-Riemannian structure here; see \cite{warhurstjet} for further details. Let  $f(x)$ and $g(x)$ be two smooth curves. We define an equivalence relation on the set of germs of smooth curves at $x_0 \in \R$ as follows:
$$ f(x) \sim_{x_0} g(x) \;\;\;\text{if and only if}\;\;\; \|f(x)-g(x)\| = O(|x-x_0|^{k+1}).  $$
The $k$-jet space is the set of equivalence classes  
$$ \mathcal{J}^k(\R,\R^n) = \bigcup_{x\in \R} C^\infty(\R,\R^n)/_{\sim_{x}}.$$
Elements of $\mathcal{J}^k(\R,\R^n)$ will be denoted by $j^k_{x}(f)$.

The space $\kJ$ admits global coordinates. An element $j^k_{x_0}(f)$ is determined by the base point $x_0$, together with the values of $f$ and its derivatives up to order $k$ at $x_0$. Thus, we write
\begin{equation*}
    \begin{split}
        j^k_{x_0}(f) =& (x_0,\mathbf{u}^k,\ldots,\mathbf{u}^0),  \;\text{where}\; \mathbf{u}^i = (u^i_1,\ldots,u^i_n),\;\; i=0,\ldots,k,\\
        & \;\text{and}\; u^i_\ell = \frac{d^if_\ell}{dx^i}(x_0),\; \; \ell=1,\ldots,n,\;\;i=0,\ldots,k . \\
    \end{split}
\end{equation*}

The $k$-jet space $\kJ$ has a natural distribution $\D$ of rank $n+1$ defined by the following Pfaffian equations:
$$ 0 = du^i_\ell - u^{i+1}_\ell dx, \;\;i=0,\ldots,k-1,\;\;\ell=1,\ldots,n. $$
The following vector fields provide a global frame for the distribution $\D$:
\begin{equation}\label{eq:lef-inv-frame}
    \begin{split}
        X & := \frac{\partial}{\partial x} + \sum_{\ell=1}^n\sum_{i=0}^{k-1}  u^{i+1}_\ell\frac{\partial}{\partial u^{i}_\ell}, \;\;
        Y_\ell  := \frac{\partial}{\partial u^{k}_\ell}, \;\text{where}\; \ell=1,\ldots,n.
    \end{split}
\end{equation}
To endow $\kJ$ with a \sR metric, we define the quadratic form
$$ds^2=\big(dx^2 +(du_1^k)^2 + \cdots+(du_n^k)^2\big)\Big|_{\D}.  $$
The canonical projection $\pi:\kJ\to \R^{n+1}$ is given in coordinates by $\pi(j_x^k(f)) = (x_,\mathbf{u}^k)$. In Appendix \ref{ap:lie-alg}, we discuss the Lie algebra and Carnot structure of $\kJ$, as well as the symmetries of $\kJ$ arising from the automorphisms of the Lie algebra.

\subsubsection{Construction of Sub-Riemannian Geodesics}\label{sss:geo-con}
We present a correspondence between $\R^n$-valued polynomial functions and \sR geodesics in $\kJ$.

Let $P_{\mu}(x):\R \to \R^n$ be a \(\mathbb R^n\)-valued polynomial of degree at most \(k\), depending on the parameter $\mu \in \R^{n(k+1)}$, i.e., 
$$P_{\mu}(x) = (P_1(x;\mu),\cdots,P_n(x;\mu)),$$
where 
\begin{equation*}\label{eq:pol-vec}
    P_\ell(x;\mu) = a^0_\ell+a^1_\ell x + \cdots + a^k_\ell x^k,\;\ell=1,\ldots,n.
\end{equation*}
Here, $\mu = (\mathbf{a}^0,\cdots,\mathbf{a}^k) \in \R^{n(k+1)}$, and $\mathbf{a}^i = (a_1^i,\cdots,a_n^i)$. The map $ \mu \to P_{\mu}$ defines a linear bijection between $\R^{(k+1)n}$ and the space of $\R^n$-valued polynomial functions of degree at most $k$. Thus, \(P_\mu\) is uniquely determined by \(\mu\), and conversely.

Let $T^*\R$ be the cotangent bundle of $\R$ with the canonical coordinates $(p_x,x)$. The vector-valued polynomial $P_{\mu}(x)$ defines the reduced Hamiltonian $H_{\mu}:T^*\R \to \R$ as follows
\begin{equation}\label{eq:red-ham}
   H_{\mu}(p_x,x) = \frac{1}{2}p_x^2 + \frac{1}{2}V_{\mu}(x),\;\text{where}\; V_{\mu}(x) = \|P_{\mu}(x)\|^2. 
\end{equation}
Here, $\|\cdot\|$ is the Euclidean norm in $\R^n$. The $x$-component evolves in an interval called the hill interval, defined below.
\begin{definition}\label{def:hill-interval}
   A hill interval for a polynomial potential $V_{\mu}(x)$ is a closed interval $I_{\mu}$ such that $V_{\mu}(x)<1$ for $x\in \Int I_{\mu}$ and $V_{\mu}(x)=1$ for $x\in \partial I_{\mu}$.
\end{definition}

    \begin{figure}[h!]
    \centering
        \begin{subfigure}{0.405\textwidth}
            \centering
            \includegraphics[width=\linewidth]{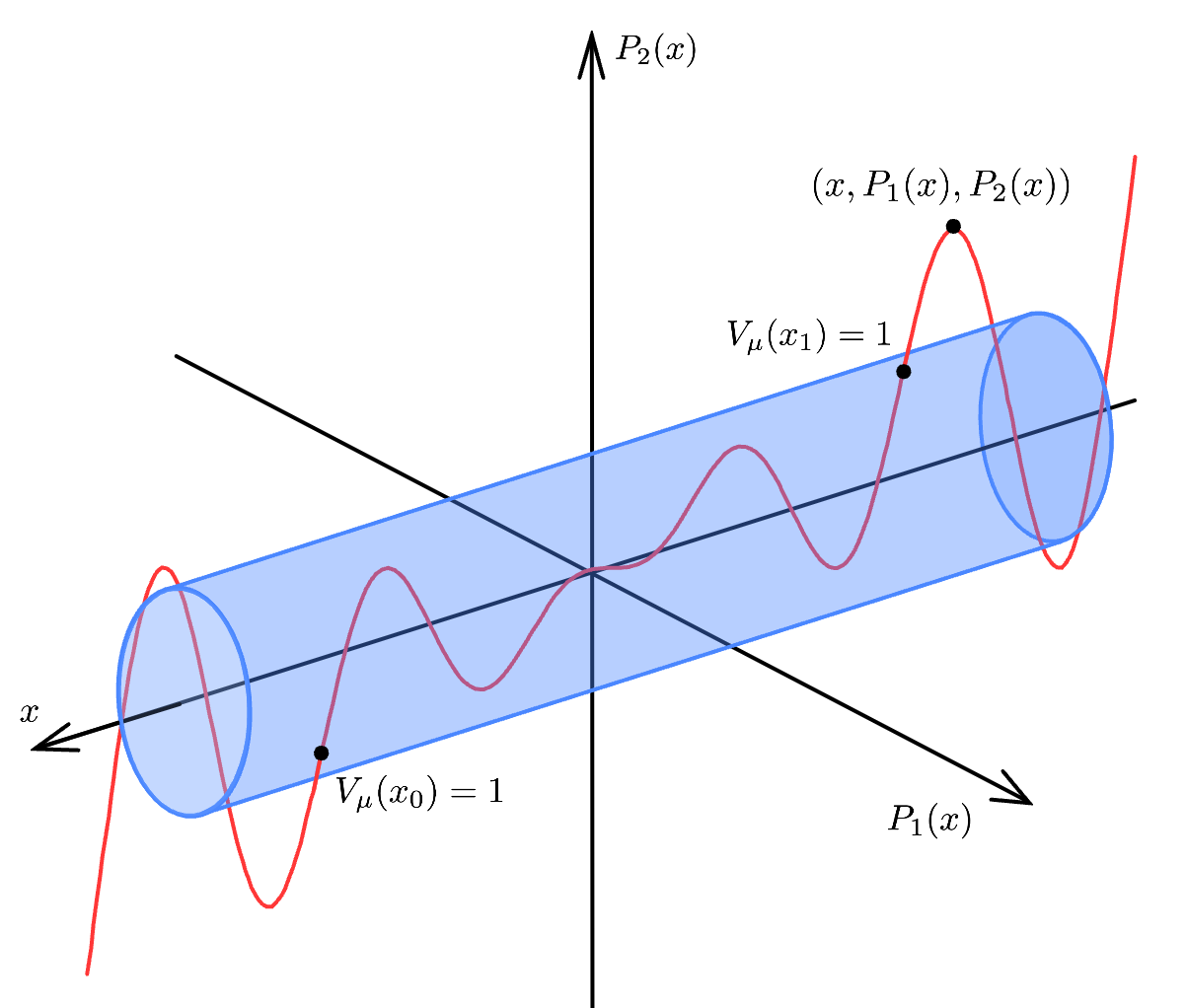}
            \caption{A hill interval \(I_\mu\) associated with the potential \(V_\mu(x)=\|P_\mu(x)\|^2\) is an interval on which the curve $\{(x,P_{\mu}(x)):x\in\R\}$ lies inside the cylinder of radius 1 around the $x$-axis.}
        \end{subfigure}\qquad
        \begin{subfigure}{0.405\textwidth}
            \centering
            \includegraphics[width=\linewidth]{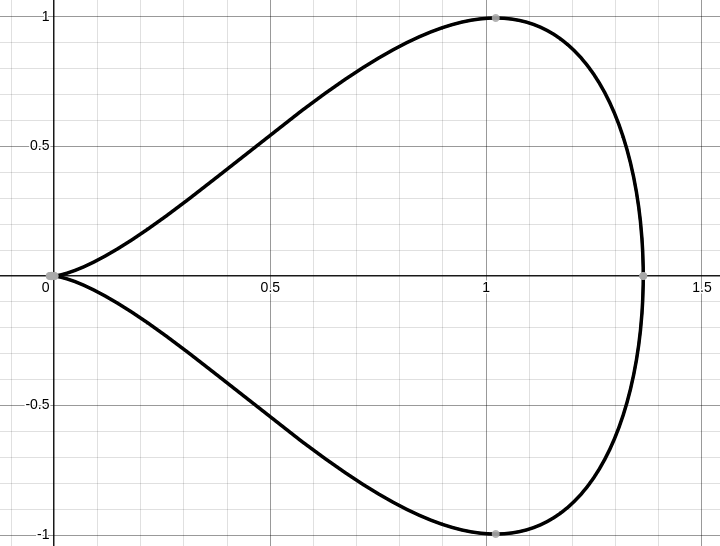}
            \caption{A singular algebraic curve in the plane $(p_x,x)$ defined by $H_\mu(p_x,x)=\frac{1}{2},$ with $x\in I_\mu$. }
        \end{subfigure}
        \caption{The panels illustrate the hill interval and the algebraic curves associated with momentum $\mu$.}
        \label{fig:hill}
    \end{figure}


We construct a normal sub-Riemannian geodesic in two steps: first, find a solution $(p_x(t),x(t))$ to the reduced Hamiltonian system satisfying the energy condition $H_{\mu}(p_x(t),x(t)) = \frac{1}{2}$; second, define a horizontal curve $\gamma(t) \in \kJ$ by solving the differential equation
\begin{equation}\label{eq:hor-lif}
     \dot{\gamma}(t) = p_x(t)X + P_1(x(t);\mu)Y_1+\cdots+ P_n(x(t);\mu)Y_n,
\end{equation}
where $\{X,Y_1,\dots,Y_n\}$ is the frame of left-invariant vector fields  defined in Eq. \eqref{eq:lef-inv-frame}.
The condition $H_{\mu} = \frac{1}{2}$ implies that $\gamma(t)$ is parametrized by arc length. Moreover, completeness of the reduced Hamiltonian system implies that $\gamma(t)$ is defined for all  $t \in \R$. The following result states that $\gamma(t)$ is a \sR geodesic. 

\begin{thm}[{\cite[Theorem 1.2]{BravoDoddoli2024}}]
   The above construction yields a normal geodesic in $\kJ$ parametrized by arc length and with momentum $\mu$. Conversely, every normal geodesic in $\mathcal{J}^k(\R,\R^n)$ parametrized by arc length arises from this construction for some vector-valued polynomial $P_{\mu}(x)$ as defined above.
\end{thm}
The observations in \cite[Remark 2.2]{BravoDoddoli+2024} concerning the analogous theorem for the jet space of functions $\mathcal{J}^k(\R,\R)$ also apply in our setting.

The Hamiltonian $H_{\mu}$ is obtained by symplectic reduction of the geodesic flow; see Appendix \ref{ap:lie-alg}. Throughout the paper, we refer to this Hamiltonian system as the reduced system or, equivalently, the reduced dynamics. The classification of sub-Riemannian geodesics employed in the sequence method is based on the behavior of their reduced dynamics. To this end, we now characterize the system's equilibrium points.
\begin{proposition}\label{prp:red-dy-eq}
    Consider a point $(p_x,x) \in H_{\mu}^{-1}(\frac{1}{2})$. Then $(p_x,x)$ is an equilibrium point of the Hamiltonian system if and only if $p_x = 0$ and $x$ satisfies $\|P_{\mu}(x)\| = 1$ and  $P_{\mu}(x)\perp P_{\mu}'(x)$. 
\end{proposition}
\begin{proof}
    A point $(p_x,x) \in H_{\mu}^{-1}(\frac{1}{2})$ is an equilibrium point if and only if  
    $$(0,0) =\nabla H_{\mu} = (p_x,P_{\mu}(x) \cdot P'_{\mu}(x)),$$ where $P'_{\mu}(x) = \frac{d}{dx} P_{\mu}(x)$ and $v_1\cdot v_2$ is the Euclidean inner product on $\R^n$. The condition $(0,0) =\nabla H_{\mu}$ implies $p_x = 0$. Hence, $(p_x,x) = (0,x) \in H_{\mu}^{-1}(\frac{1}{2})$ if and only if $\|P_{\mu}(x)\| =1$. Thus, \((p_x,x)\) is an equilibrium if and only if
    $p_x=0$, $\|P_\mu(x)\|=1$, and  $P_\mu(x)\perp P_\mu'(x)$.
\end{proof}

The equilibrium points of the reduced Hamiltonian system correspond to the singular points of the algebraic curve $H_{\mu}^{-1}(\frac{1}{2})$. This observation plays a fundamental role in Proposition \ref{prp:sin-point}; see also Fig. \ref{fig:hill}.

\subsubsection{Classification of \SR Geodesics}\label{sub-sec:class-geo}
We now classify the sub-Riemannian geodesics by their reduced dynamics. According to the classical theory of one-degree-of-freedom systems, the sub-Riemannian geodesics fall into one of the following types:

\textbf{Line geodesics:} A geodesic $\gamma(t)$ is of line type if $\pi(\gamma(t))$ is a straight line in $\R^{n+1}$. Line geodesics correspond to a constant polynomial $P_{\mu}(x) = \textbf{c}$. If $\|\textbf{c}\| < 1$, then $I_{\mu} = \R$, whereas if $\|\textbf{c}\| = 1$, then $I_{\mu}$ is a singleton.

\textbf{Periodic geodesics:} A geodesic $\gamma(t)$ is of periodic type with hill interval $I_{\mu}$ if the reduced dynamics are periodic. The reduced dynamics are periodic if and only if $H_{\mu}^{-1}(\frac{1}{2})|_{I_{\mu}}$ is smooth.  

\textbf{Homoclinic geodesics:} A geodesic $\gamma(t)$ is of homoclinic type with hill interval $I_{\mu}$ if the reduced dynamics have a homoclinic orbit. The reduced dynamics have a homoclinic orbit if and only if $H_{\mu}^{-1}(\frac{1}{2})|_{I_{\mu}}$ contains exactly one singular point.

\textbf{Heteroclinic geodesics:} A geodesic $\gamma(t)$ is of  heteroclinic type with hill interval $I_{\mu}$ if the reduced dynamics have a heteroclinic orbit. The reduced dynamics have a heteroclinic orbit if and only if $H_{\mu}^{-1}(\frac{1}{2})|_{I_{\mu}}$ contains two singular points.

 Eq. \eqref{eq:hor-lif} implies that every homoclinic geodesic satisfies the condition stated in  Conjecture \ref{conj}. However, this is not true for every heteroclinic geodesic. The following definitions distinguish the two cases.

\textbf{Direct-type geodesics:} A heteroclinic geodesic $\gamma(t)$ is of direct type  with hill interval $I_{\mu} = [x_0,x_1]$ if and only if  $ P_{\mu}(x_0) = P_{\mu}(x_1)$.

\textbf{Turn-back geodesics:} A heteroclinic geodesic $\gamma(t)$ is of turn-back type with hill interval $I_{\mu} = [x_0,x_1]$ if and only if  $ P_{\mu}(x_0) \neq  P_{\mu}(x_1)$.

    \begin{figure}[h!]
    \centering
        \begin{subfigure}{0.38\textwidth}
            \centering
            \includegraphics[width=\linewidth]{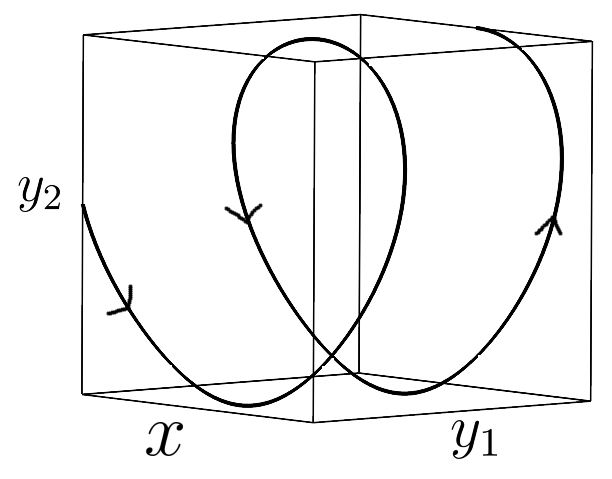}
            \caption{Periodic.}
        \end{subfigure}\qquad
        \begin{subfigure}{0.34\textwidth}
            \centering
            \includegraphics[width=\linewidth]{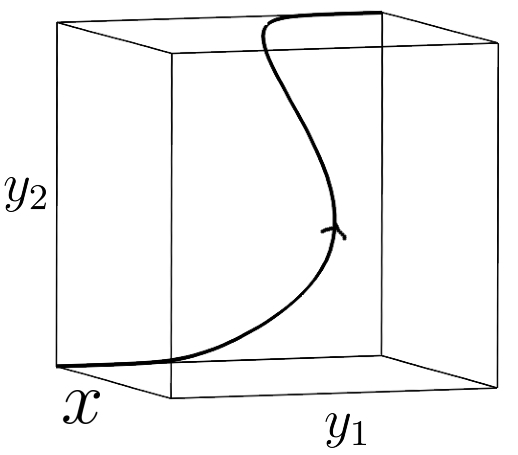}
            \caption{Homoclinic. }
        \end{subfigure}
        \qquad
        \begin{subfigure}{0.34\textwidth}
            \centering
            \includegraphics[width=\linewidth]{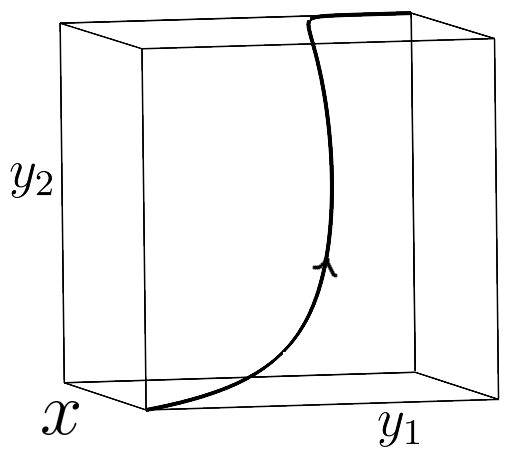}
            \caption{Direct-type. }
        \end{subfigure}
        \qquad
        \begin{subfigure}{0.38\textwidth}
            \centering
            \includegraphics[width=\linewidth]{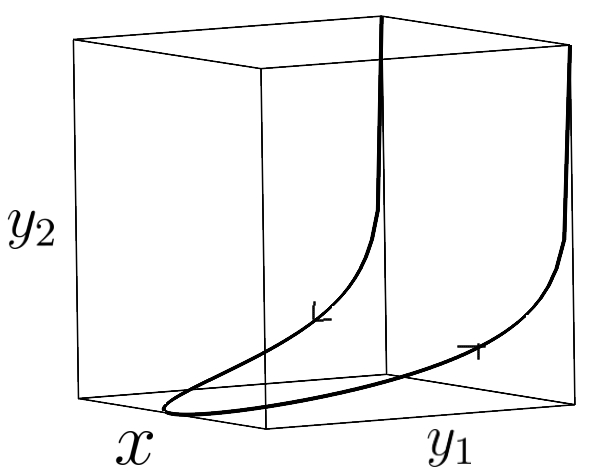}
            \caption{Turn-back. }
        \end{subfigure}
        \caption{The panels display the projections of the different types of geodesics in $\PJ$  onto $\R^3$, with coordinates $(x,y^1,y^2)$.}
        \label{fig:geosic examples}
    \end{figure}



\begin{proposition}\label{prp:abn-geo}
  Let $ k\geq 2$. The singular curves in $\kJ$ are precisely the curves $\gamma(t)$ such that
  $$ \dot{\gamma}(t) \in \mathrm{span}\{ Y_1,\ldots,Y_n\}\;\text{whenever}\;\dot{\gamma}(t)\;\text{exists}.$$
  Consequently, the abnormal geodesics in $\kJ$ are line geodesics whose $x$-component is constant. 
\end{proposition}


\subsection{The Two-Jet Space of Plane Curves}\label{subsec:two-jet}

For the remainder of the paper, we focus on $\PJ$. We begin by classifying its geodesics whose reduced algebraic curves are singular.
\begin{proposition}\label{prp:only-hom}
The only geodesics whose reduced algebraic curves are singular in $\PJ$ are homoclinic.
\end{proposition}
\begin{proof}
    The potential function $V_{\mu}(x)$ associated with a geodesic in $\PJ$ has degree at most 4. However, if $V_{\mu}(x)$ is associated with a heteroclinic geodesic, then $V_{\mu}(x)$ must have degree at least 5. Indeed, if $V_{\mu}(x)$ is a potential associated with a heteroclinic geodesic and with hill interval $I_{\mu} = [x_0,x_1]$, then $V_{\mu}(x)$ must admit the factorization $1-V_{\mu}(x) = (x-x_0)^{r_0}(x-x_1)^{r_1}q(x)$, where $r_0,r_1\geq 2$. If $ r_0+ r_1 \geq 5$, the claim follows immediately. It remains to consider the case $r_0=r_1=2$. Since \(V_\mu\) has degree at most \(4\), we must have
    $$ 1-V_{\mu}(x) = C(x-x_0)^{2}(x-x_1)^{2} $$
    for some constant \(C\). Since \(V_\mu(x)<1\) for \(x\in(x_0,x_1)\), we have \(C>0\). Thus, the leading coefficient of \(V_\mu\) is negative. On the other hand, \(V_\mu=P_1^2+P_2^2\), so its leading coefficient is nonnegative, a contradiction.
\end{proof}   

To classify metric lines up to isometries, we use the symmetries described in  Appendix \ref{subsub:sym} to reduce the number of parameters. First, by applying Carnot translations and dilations, we restrict our attention to homoclinic geodesics with hill interval $I_{\mu} = [0, x_1]$. Second, using rotational symmetry, we fix the asymptotic velocity as $v := (0,1,0)$. 

We denote by $\Lambda\subset\R^6$ the set of momenta $\mu$ for which the associated \(\mathbb R^n\)-valued polynomial $P_{\mu}(x)$ satisfies the following properties: 
    \begin{itemize}
        \item  $I_{\mu} = [0,x(\mu)]$ is a hill interval, where   $x(\mu)>0$.

        \item $P_{\mu}(0) = (1,0)$ and $P_{\mu}(0)\cdot P'_{\mu}(0) = 0$.
    \end{itemize}
Let 
\begin{equation*}
    \begin{split}
       \Int \mathcal{A} & := \{ (\lambda,a,b) \in (0,\infty)\times\R^2: 0 < 2 -b^2 \}, \\
        \partial \mathcal{A} & := \{ (\lambda,a,b) \in (0,\infty)\times\R^2 : b = \pm\sqrt{2},\;\; ab<0  \},
    \end{split}
\end{equation*}
and set $A:=\operatorname{Int}A\cup\partial A$.

Define the map $\mu:\mathcal{A} \to \Lambda$ by 
\begin{equation}\label{eq:mu-Lambda}
    \mu(\lambda,a,b) = (1,0,0, \frac{b}{\lambda}, \frac{-1}{\lambda^2},\frac{a}{\lambda^2}) := (a^0_1,a^0_2,a^1_1,a^1_2,a^2_1,a^2_2) \in \Lambda .
\end{equation}
 The map $\mu$ is a bijection. Indeed, $\mu$ is clearly injective. To see that it is surjective, note that the condition $P_{\mu}(0) = (1,0)$ implies $(a_1^0,a_2^0) = (1,0)$. The condition  $P_{\mu}(0)\cdot P'_{\mu}(0) = 0$ implies $a_1^1 = 0$.
 Since $I_\mu = [0,x(\mu)]$ is a hill interval with $x(\mu)>0$, we have \(V_\mu(x)<1\) for \(x\in(0,x(\mu))\). In particular, \(P_1(x;\mu)^2<1\). Since \(P_1(x;\mu)=1+a_1^2x^2\), it follows that \(a_1^2<0\).  By setting $a_1^2 = \frac{-1}{\lambda^2}$, $a_2^1 = \frac{b}{\lambda}$, and $a_2^2 = \frac{a}{\lambda^2}$, we find that $x(\mu)>0$ if and only if $(\lambda,a,b) \in \mathcal{A}$.

%
\subsubsection{The Candidate Geodesics}\label{subsubsec:cand-geo}

Here, we introduce the two families of homoclinic geodesics mentioned in Theorem \ref{the:main} by specifying two subsets of $\Lambda$ using Eq. \eqref{eq:mu-Lambda}.

\textbf{ \textit{Family 1}:} Let $\mathcal{A}_1 := \mathcal{A}\cap\{b=0\}$ and $\Lambda_1 := \mu(\mathcal{A}_1)$. \textit{Family 1} consists of the geodesics corresponding to momenta $\mu \in \Lambda_1$.

\textbf{ \textit{Family 2}:} Let $\mathcal{A}_2 := \mathcal{A}\cap\{a=0\}$ and $\Lambda_2 := \mu(\mathcal{A}_2)$. \textit{Family 2} consists of the geodesics corresponding to momenta $\mu \in \Lambda_2$.

The remainder of the paper is devoted to proving Theorem \ref{thm:main-3}, which states that, for $i=1,2$, every homoclinic geodesic with momentum  $\mu \in \Lambda_i$ is a metric line. 

Before continuing, we treat separately the case $a=b=0$, corresponding to  $\mathcal{A}_1\cap\mathcal{A}_2$. The Carnot group $\PJ$ contains a normal subgroup $\mathcal{N}$ such that $\PJ/\mathcal{N} \simeq \mathcal{J}^2(\R,\R)$; see Appendix \ref{ap:lie-alg} for the explicit construction of this subgroup. It follows that the canonical projection $\pi_{\mathcal{N}}:\PJ\to J^2(\R,\R)$ is a submetry by \cite[Corollary 6.3.5]{le2025metric}. Moreover, if $\gamma(t)$ is a geodesic with momentum $\mu \in \mu(\mathcal{A}_1\cap\mathcal{A}_2)$, then $\pi_{\mathcal{N}}(\gamma(t))$ is a geodesic in $\mathcal{J}^2(\R,\R)$. This geodesic is a metric line in $\mathcal{J}^2(\R,\R)$; see \cite[Theorem 1.4]{BravoDoddoli+2024}. Therefore, Lemma \ref{lemm:hor-lift-met-lin} implies that $\gamma(t)$ is a metric line in $\PJ$. We define
$$ \Lambda_0 := \mu(\mathcal{A}_1\cap\mathcal{A}_2) = \Big\{  (1,0,0,0,\frac{-1}{\lambda^2},0): \lambda >0\Big\}.$$
For the remainder of the paper, we restrict our attention to geodesics with momentum $\mu\in \Lambda\setminus\Lambda_0$.

\subsection{Magnetic Space}\label{subsec:mag-space}

Let $\gamma_{h}(t)$ be a homoclinic geodesic with momentum $\mu\in \Lambda\setminus \Lambda_0$. By applying a Carnot dilation, we may assume, without loss of generality, that $\lambda = 1$. Thus, for the remainder of the paper, we consider the vector-valued polynomial
$$ P_{(a,b)}(x) := (P_1(x),P_2(x)) = (1-x^2,bx+ax^2),\;\;\text{with}\;\;(a,b) \neq (0,0). $$
Since the normalized momentum \(\mu(1,a,b)\) is uniquely determined by \((a,b)\), we henceforth write \(P_i(x)\) instead of \(P_i(x;\mu)\) and omit \(\mu(1,a,b)\) from the notation.

We now construct a sub-Riemannian magnetic space $\R^5_{(a,b)}$ whose geometry depends on the vector-valued polynomial $P_{(a,b)}(x)$. We also construct a \sR submersion $\pi_{(a,b)}:\PJ \to \R^5_{(a,b)}$ such that \(\pi\) factors as $\pi  = pr \circ \pi_{(a,b)}$, where $pr:\R^5_{(a,b)} \to \R^{3}$ is a submersion. Therefore, by Lemma \ref{lemm:hor-lift-met-lin}, it suffices to prove that the \sR geodesic $c_h(t):=\pi_{(a,b)}(\gamma_{h}(t))$ is a metric line in $\R^5_{(a,b)}$. We call $c_h(t)$ the candidate geodesic.

\begin{thm}\label{thm:main-3}
    For $i=1,2$, consider $\mu(1,a,b) \in \Lambda_i\setminus \Lambda_0$, and let $\gamma_h(t)$ be a homoclinic geodesic with momentum $\mu(1,a,b)$. Then $c_h(t) := \pi_{(a,b)}(\gamma_h(t))$ is a metric line in $\R^5_{(a,b)}$. Consequently, $\gamma_h(t)$ is a metric line in $\PJ$.  
\end{thm}

\subsubsection{The Magnetic Space}\label{sss:mag-spa}
The magnetic space $\R^5_{(a,b)}$ associated with $\PJ$ has the following \sR structure. 


Consider $\R^5_{(a,b)}$ with coordinates $(x,y^1,y^2,z^1,z^2)$. For convenience, write $\textbf{y} = (y^1,y^2)$ and $\textbf{z} = (z^1,z^2)$. We define a bracket-generating distribution $\D_{(a,b)}$ on $\R^5_{(a,b)}$ by
$$ dz^1 - P_1(x) dy^1 = 0,\;\; dz^2 - P_2(x) dy^2 = 0. $$
Hence, the following vector fields provide a global frame for  $\D_{(a,b)}$
$$ \widetilde{X} = \frac{\partial}{\partial x}, \;\;\widetilde{Y}_1 = \frac{\partial}{\partial y^1} + P_1(x)\frac{\partial}{\partial z^1},\;\;\text{and} \;\;\widetilde{Y}_2 = \frac{\partial}{\partial y^2} + P_2(x)\frac{\partial}{\partial z^2}.  $$
The quadratic form $ ds^2_M = (dx^2 + (dy^1)^2+ (dy^2)^2)|_{\D_{(a,b)}}$ induces a \sR inner product on $\D_{(a,b)}$. The \sR submersion $\pi_{(a,b)}$ is  given by
$$\pi_{(a,b)}(j^2_{x}(f)) = (x,u^2_1,u^2_2, \sum_{i=0}^2 (-1)^iu^{2-i}_1\frac{d^i P_1}{dx^i},\sum_{i=0}^2 (-1)^iu^{2-i}_2\frac{d^i P_2}{dx^i}). $$


\subsubsection{\SR Geodesics in the Magnetic Space}\label{sss:geo-mag}
Consider the cotangent bundle $T^*\R^5_{(a,b)}$ with canonical coordinates $(p_x,\mathbf{p}_{\mathbf{y}},\mathbf{p}_{\mathbf{z}},x,\mathbf{y},\mathbf{z})$.  The \sR kinetic energy $H_{(a,b)}:T^*\R^5_{(a,b)} \to \R$ is given by 
$$ H_{(a,b)}(p_x,\mathbf{p}_{\mathbf{y}},\mathbf{p}_{\mathbf{z}},x,\mathbf{y},\mathbf{z}) = \frac{1}{2}p_x^2 + \frac{1}{2}(p_{y^1}+p_{z^1}P_1(x))^2+  \frac{1}{2}(p_{y^2}+p_{z^2}P_2(x))^2.$$
The Hamiltonian function $H_{(a,b)}$ does not depend on the coordinates $(\mathbf{y},\mathbf{z})$, so  the momentum $\nu :=(\mathbf{p}_{\mathbf{y}},\mathbf{p}_{\mathbf{z}})$ is conserved along the Hamiltonian flow. Moreover, the translation $\Phi_{(\mathbf{y}_0,\mathbf{z}_0)}:\R^5_{(a,b)} \to \R^5_{(a,b)}$  given by
\begin{equation}\label{eq:tran-isom}
    \Phi_{(\mathbf{y}_0,\mathbf{z}_0)}(x,\mathbf{y},\mathbf{z}) = (x,\mathbf{y}+\mathbf{y}_0,\mathbf{z}+\mathbf{z}_0),
\end{equation}
is an isometry. 

As in the case of $\PJ$, the \sR geodesics in $\R^5_{(a,b)}$ are in bijection with a pencil of \(\mathbb R^n\)-valued polynomials. The following definitions formalize this fact.
\begin{definition}
     The pencil \(\operatorname{Pen}_{(a,b)}\) of \(P_{(a,b)}(x)\) is the four-dimensional vector space of vector-valued polynomials
    $$ G_\nu(x) := (G_1(x;\nu),G_2(x;\nu)) = (\alpha_1+\beta_1P_1(x),\alpha_2+\beta_2P_2(x)),$$
    where $\nu = (\alpha_1,\alpha_2,\beta_1,\beta_2) \in \R^4$.
\end{definition}

We now describe the \sR geodesics in $\R^5_{(a,b)}$, their horizontal lifts, and their relation to sub-Riemannian geodesics in $\PJ$.

A \sR geodesic $c(t)$ in $\R^5_{(a,b)}$ corresponds to a vector-valued polynomial $G_{\nu}(x)$, and  its $x$-component satisfies a one-degree-of-freedom system given by the Hamiltonian function
    \begin{equation}
        H_\nu(p_x,x) = \frac{1}{2}p_x^2 + \frac{1}{2}V_{\nu}(x),\;\;\text{where}\;\;V_\nu(x) = \|G_\nu(x)\|^2. 
    \end{equation}
Once $x(t)$ is determined from the above Hamiltonian system, the coordinates $\mathbf{y}$ and $\mathbf{z}$ satisfy the equations
    \begin{equation}\label{eq:y-z-ode}
        \dot{y}_i = G_i(x(t);\nu),\;\;\dot{z}_i = G_i(x(t) ;\nu) P_i(x(t)),\;\;i=1,2.
    \end{equation}
    The horizontal lift of the geodesic $c(t)$ is the solution to the differential equation
    \begin{equation*}
        \dot{\gamma}(t) = p_x(t)X(\gamma(t)) + G_1(x(t);\nu) Y_1(\gamma(t))+G_2(x(t);\nu) Y_2(\gamma(t)).
    \end{equation*}
The preceding discussion proves the following lemma.
\begin{lemma}\label{the:Rab-geo}
Every \sR geodesic in $\R^5_{(a,b)}$ is the projection under $\pi_{(a,b)}$ of a geodesic in $\PJ$ corresponding to a vector-valued polynomial $G_{\nu}(x) \in \mathrm{Pen}_{(a,b)}$. Conversely, the horizontal lift of a geodesic in $\R^5_{(a,b)}$ is a geodesic in $\PJ$ corresponding to a vector-valued polynomial in $\mathrm{Pen}_{(a,b)}$.
\end{lemma}
The \sR geodesic $c_h(t) $ in $\R^5_{(a,b)}$  has momentum $\nu(1,1)= (0,0,1,1)$.

The $x$-component of the reduced dynamics has hill interval $I_\nu$.  A vector-valued polynomial $G_{\nu}(x) \in \mathrm{Pen}_{(a,b)}$ defines a reduced algebraic curve given by
\begin{equation}\label{eq:alg-cur-G}
    \C(\nu,I_{\nu}) := \{ (p_x,x) : 1 = p_x^2 + V_\nu(x)\;\text{and}\;x\in I_\nu \}.
\end{equation}
The algebraic curve $\C(\nu,I_{\nu})$ will play an important role later when we analyze the smoothness of the period map and other functions. The relation between the singular points of $\C({\nu},I_{\nu})$ and the equilibrium points of the Hamiltonian system for $H_{\nu}$ is given by the following result. 

\begin{proposition}\label{prp:sin-point}
    Consider $(p_x,x) \in \C(\nu,I_{\nu})$. The point $(p_x,x)$ is a singular point of $\C(\nu,I_{\nu})$ if and only if $(p_x,x)$ is an equilibrium point of the Hamiltonian system for $H_{\nu}$. 
\end{proposition}
\begin{proof}
The equilibrium points of the Hamiltonian system for $H_{\nu}$ are characterized as in Proposition \ref{prp:red-dy-eq}. Thus, $(p_x,x)$ is an equilibrium point if and only if $\nabla H_{\nu} = (0,0)$, which is precisely the condition for $(p_x,x)$ to be a singular point of $\C(\nu,I_{\nu})$. 
\end{proof}

We classify the sub-Riemannian geodesics in $\R^5_{(a,b)}$ according to their reduced dynamics, as in Section \ref{sub-sec:class-geo}. In particular, the candidate geodesic $c_h(t)$ is homoclinic.  

\subsubsection{The Set of Homoclinic Geodesics}\label{sss:hom-goe}
Let $\Lambda_{(a,b)}$ denote the set of momenta corresponding to homoclinic geodesics with the same asymptotic velocity as the candidate geodesic $c_h(t)$. We write
$$\Lambda_{(a,b)} :=\Lambda^{+}_{(a,b)}\cup \Lambda^{-}_{(a,b)} \subset\R^4.$$
 More precisely, $\nu \in \Lambda_{(a,b)}$ if and only if the vector-valued polynomial $G_{\nu}(x)$ satisfies
 \begin{itemize}
 \item $G_{\nu}(0) = (1,0)$ and $G_{\nu}(0)\cdot G_{\nu}'(0) = 0$.
 \end{itemize}
The superscript distinguishes the two cases:
 \begin{itemize}
    \item If $\nu \in \Lambda^{+}_{(a,b)}$, then the hill interval has the form $I^+_{\nu} = [0,x^+(\nu)]$, where $ x^+(\nu)>0$;
    
    \item If $\nu \in \Lambda^{-}_{(a,b)}$, then the hill interval has the form $I^-_{\nu} = [x^-(\nu),0]$, where $x^-(\nu) < 0$.
 \end{itemize}
 We define 
\begin{equation}\label{eq:Acal-1}
    \mathcal{A}^+_{(a,b)}  :=  \begin{cases}
       \{ (\tau,\eta) \in (0,\infty)\times\R:  2\tau -b^2\eta^2 >0\},\quad \text{if}\;  ab\geq 0,\\
       \{ (\tau,\eta) \in (0,\infty)\times\R:  2\tau -b^2\eta^2 \geq 0\},\quad \text{if}\;  ab<0.
    \end{cases}
\end{equation}
 We define 
\begin{equation}\label{eq:Acal-2}
    \mathcal{A}^-_{(a,b)} :=  \begin{cases}
         \{ (\tau,\eta) \in (0,\infty)\times\R: 2\tau -b^2\eta^2 \geq 0\}, \quad \text{if}\;  ab>0,\\
         \{ (\tau,\eta) \in (0,\infty)\times\R:  2\tau -b^2\eta^2 > 0\}, \quad \text{if}\;  ab\leq 0.
    \end{cases}
\end{equation}
Define the map $\nu: \mathcal{A}^\pm_{(a,b)} \to \Lambda^\pm_{(a,b)}$ by 
\begin{equation}\label{eq:nu-par}
    \nu(\tau,\eta) = (1-\tau,0,\tau,\eta).
\end{equation} 
The map $\nu$ is a bijection.

Therefore, the vector-valued polynomial $G_{\nu}(x)$ has the form 
\begin{equation}
    G_{\nu}(x) = (G_1(x;\nu),G_2(x;\nu)) = (1-\tau x^2,\eta(bx+ax^2)),
\end{equation}
and the candidate geodesic $c_h(t)$ has momentum $\nu(1,1)$. Moreover, $x^{\pm}(\nu)$ is given by
\begin{equation}\label{eq:root-nu}
   x^{\pm}(\nu) = \frac{-ab\eta^2 \pm \sqrt{\tau^2(2\tau-b^2\eta^2)+2a^2\eta^2\tau}}{\tau^2+a^2\eta^2} .
\end{equation}

The following result characterizes the abnormal geodesics in $\R^5_{(a,b)}$.
\begin{proposition}\label{prp:abn-mag-spa}
  The magnetic space $\R^5_{(a,b)}$ admits a family of abnormal geodesics characterized by
  $$ \dot{c}(t) = v_1\tilde{Y}_1+v_2\tilde{Y}_2, $$
  where $(v_1,v_2)$ is a constant vector satisfying 
  $$0 = (v_1,v_2)\cdot (c_1\frac{dP_1}{dx} ,c_2\frac{dP_2}{dx})$$ 
  with $x$ constant and $(c_1,c_2) \in \R^2\setminus\{(0,0)\}$. Consequently, there are no strictly abnormal geodesics in $\R^5_{(a,b)}$. 
\end{proposition}
The proof of Proposition \ref{prp:abn-mag-spa} is given in Appendix \ref{Ap:abn}.

\subsubsection{The Cost Function}\label{sss:cost-fuc}
The cost function is an auxiliary tool used to prove Theorem \ref{thm:main-2} and to define the period map.

   The following two sets play a fundamental role in this section. 
\begin{equation}
\begin{split}
    \mathcal{C} & := \{c(t) \in \R^{5}_{(a,b)} : \; c(t) \; \text{is a geodesic parametrized by arc length} \}, \\
    \mathcal{I} & := \{\T = [t_0,t_1] \subset \R : \;  t_0<t_1\}. \\
\end{split}
\end{equation}
    We denote by $\Delta: \mathcal{C} \times \mathcal{I} \to [0,\infty) \times \R^4$ the map given by
    \begin{equation}
    \begin{split}
    \Delta(c,\T) & := (\Delta t(c,\T), \Delta \mathbf{y}(c,\T) , \Delta \mathbf{z}(c,\T))\\
    & := (t_1-t_0, \mathbf{y}(t_1) - \mathbf{y}(t_0), \mathbf{z}(t_1) - \mathbf{z}(t_0)).
    \end{split}
    \end{equation}
    The map $\Delta (c,\T)$ records the elapsed time and the changes in the coordinates $\mathbf{y}$ and $\mathbf{z}$ as the geodesic $c(t)$ evolves over $\T$. In addition, consider $v :=(0,\hat{v})$ where $\hat{v} = (v_1,v_2) \in \R^2$ is a unit vector. We denote by $\mathrm{Cost}^v: \mathcal{C} \times \mathcal{I} \to [0,\infty) \times \R$ the map given by
    $$ \mathrm{Cost}^v(c,\T):=(\mathrm{Cost}_t^v(c,\T),\mathrm{Cost}_y^v(c,\T)) ,$$
    where
    \begin{equation*}
        \begin{split}
           \mathrm{Cost}_t^v(c,\T) &= \Delta t(c,\T) -    \hat{v} \cdot\Delta \mathbf{y}(c,\T) , \\
        \mathrm{Cost}_y^v(c,\T) &= \hat{v} \cdot (\Delta \mathbf{y}(c,\T) -  \Delta \mathbf{z}(c,\T)). \\
        \end{split}
    \end{equation*}
    We call $\mathrm{Cost}_t^v(c,\T)$ the cost of \(c\) over $\T$.

For the remainder of this subsection, we discuss the significance of the cost function. In particular, $\mathrm{Cost}_t^v(c,\T)$ measures the cost incurred by $c(t)$ when traveling in the direction $v$. The following result provides further insight into this interpretation.

\begin{proposition}\label{prp:cost-short}
    Let $\mathbf{A},\mathbf{B}\in \R^5_{(a,b)}$. Let $c$ and $\widetilde{c}$ be two \sR geodesics from $\mathbf{A}$ to $\mathbf{B}$, defined on the intervals $\T$ and $\widetilde{\T}$, respectively. If $\mathrm{Cost}_t^v(c,\T) < \mathrm{Cost}_t^v( \widetilde{c},\widetilde{\T})$, then the arc length of $c$ is shorter than that of $\widetilde{c}$.
\end{proposition}

\begin{proof}
   Let $c,\widetilde{c}\in \mathcal{C}$ be two \sR geodesics. Since $c,\widetilde{c}\in\mathcal{C}$ travel from $\mathbf{A}$ to $\mathbf{B}$, we have $\Delta \textbf{y}(c,\T) = \Delta \textbf{y}(\widetilde{c},\widetilde{\T})$. If $\mathrm{Cost}^v_t(c,\T) < \mathrm{Cost}^v_t(\widetilde{c},\widetilde{\T})$, we have
    \begin{equation*}
        \begin{split}
    \Delta t(c,\T) = &  \;\mathrm{Cost}^v_t(c,\T) + \hat{v} \cdot \Delta \textbf{y}(c,\T) \\
    & < \mathrm{Cost}^v_t( \widetilde{c},\widetilde{\T}) + \hat{v} \cdot \Delta \textbf{y}(\widetilde{c},\widetilde{\T})   =    \Delta t(\widetilde{c},\widetilde{\T}).  
        \end{split}
    \end{equation*}  
    Since $c$ and $\widetilde{c}$ are parametrized by arc length, their lengths over $\T$ and $\widetilde{\T}$ are $\Delta t(c,\T)$ and $\Delta t(\widetilde{c},\widetilde{\T})$, respectively. The result follows.
\end{proof}

 To simplify the notation, define the non-exact one-form $\phi_{\nu}$ on $\mathcal{C}(\nu,I_{\nu})$ by
\begin{equation*}
    \phi_{\nu} := \frac{dx}{p_x}.
\end{equation*}    
The one-form  $\phi_{\nu}$ is smooth if and only if $\C(\nu,I_{\nu})$ is smooth. Indeed, if $I_{\nu} = [x_0,x_1]$, then $\C( \nu,I_{\nu})$ can be locally parametrized by $x$ for $x \in (x_0,x_1)$, namely $(p_x,x) = (\pm \sqrt{1-V_{\nu}(x)},x)$. However, at the points $x = x_0,x_1$, the implicit function theorem provides a parametrization of $\C( \nu,I_{\nu})$ in terms of $p_x$. Indeed, the implicit derivative is given by
$$ \frac{\partial x}{\partial p_x} = -\frac{2 p_x}{V'_{\nu}(x)},\;\;\text{which implies that} \;\;  \phi_{\nu} = -\frac{2 dp_x}{V'_{\nu}(x)}. $$
Thus, $\phi_{\nu}$ is  smooth at $(0,x_0)$ if and only if $V'_{\nu}(x_0) = 2G_{\nu}(x_0) \cdot G'_{\nu}(x_0) \neq 0$. On the other hand, $(0,x_0) \in \C(\nu,I_{\nu})$ is an equilibrium
point of the Hamiltonian system associated with $H_{\nu}$ if and only if $G_{\nu}(x_0) \cdot G'_{\nu}(x_0) = 0$. By Proposition \ref{prp:sin-point}, $(0,x_0)$ is an equilibrium point if and only if $(0,x_0)$ is a singular point of $\C(\nu,I_{\nu})$. The same argument applies to $(0,x_1).$  Therefore, $\phi_{\nu}$ is smooth if and only if  $\C(\nu,I_{\nu})$ is smooth.

\begin{proposition}\label{pro:cost-G}
    Let $c(t) \in \R^5_{(a,b)}$ be a geodesic with momentum $\nu$, and let $\T \in \mathcal{I}$. The map  $\Delta(c,\T)$ can be expressed in terms of the vector-valued polynomial $G_{\nu}(x)$ as
    \begin{equation*}
    \begin{split}
    \Delta t(c,\T) & =  \int_{x(\T)} \phi_{\nu},\\
        \Delta y_i(c,\T)  &= \int_{x(\T)} G_i(x;\nu)\phi_{\nu},\;\; i=1,2,\\
        \Delta z_i(c,\T) & = \int_{x(\T)} G_i(x;\nu)P_i(x)\phi_{\nu}, \;\; i=1,2. \\
    \end{split}
    \end{equation*}
    Here, $x(\T)$ denotes the oriented image of $\T$ under $x$, with orientation induced by increasing $t$.

     Similarly, the cost functions can be expressed in terms of $G_\nu(x)$ as
     \begin{equation*}
    \begin{split}
        \mathrm{Cost}_t^v(c,\T) & = \int_{x(\T)} (1  - \hat{v}\cdot G_{\nu}(x))\phi_{\nu},\\
        \mathrm{Cost}_y^v(c,\T) & = \sum_{i=1}^2 v_i \int_{x(\T)} G_i(x;\nu)(1 -P_i(x)) \phi_{\nu}. \\
    \end{split}
    \end{equation*}
\end{proposition}
\begin{proof}
  Since $H_{\nu}$ has one degree of freedom, we reduce it to quadrature using the standard method; that is, we consider the energy condition $H_{\nu} = \frac{1}{2}$, and solve the algebraic equation for $p_x$ to obtain
  $$ \dot x =  p_x = \pm \sqrt{1-V_{\nu}(x)},$$
  where we use Hamilton's equation $\dot{x} = \frac{\partial H_{\nu}}{\partial p_x}$. Considering \(t\) as a function of \(x\), we obtain $dt = \phi_{\nu}$. Integrating $dt$ yields the formula for $\Delta t(c,\T)$ stated in the proposition. Using $dt = \phi_{\nu}$ and Eq. \eqref{eq:y-z-ode}, we obtain
  $$ dy^i = G_i(x;\nu)\phi_{\nu},\;\; dz^i = P_i(x)G_i(x;\nu)\phi_{\nu}.$$
  Finally, the definitions of the cost functions yield the corresponding formulas for $\mathrm{Cost}^v_t(c,\T)$ and $\mathrm{Cost}^v_y(c,\T)$.
\end{proof}

\begin{remark}\label{rem:one-form-close} We make the following observations concerning Proposition \ref{pro:cost-G}:

    \textbf{(1)} The function $\Delta t(c,\T)$ depends not only  on the endpoints $x(t_0)$ and $x(t_1)$, but also on the specific path $x(\T)$.  For instance, if $c(t)$ is a periodic geodesic, then $x(t+mL(I_\nu)) = x(t)$ for $m\in \mathbb{N}$ and
   $$ \Delta t(c,[t,t+mL(I_\nu)]) = m L(I_\nu),\;\;\text{where}\;\; L(I_\nu) := 2 \int_{I_\nu} \phi_{\nu}.$$
    Hence, $m$ is the degree of the map $t \to (p_x(t),x(t))$, where the solution is viewed as a map from $\mathbb{S} := \R/L(I_\nu)\mathbb{Z}$ to $\C(\nu,I_{\nu})$. 

    \textbf{(2)}  The functions $\Delta t(c,\T)$, $\Delta y^i(c,\T)$, and $\Delta z^i(c,\T)$  are finite whenever the corresponding one-forms are smooth along $x(\T)$. The same conclusion holds for $\mathrm{Cost}^v_t(c,\T)$ and $\mathrm{Cost}^v_y(c,\T)$.
\end{remark}

For the remainder of the paper, we restrict our attention to homoclinic geodesics with momentum $\nu \in \Lambda_{(a,b)}$. 
Accordingly, we consider only the asymptotic velocity $v=(0,1,0)$ and omit the superscript $v$ from the cost functions for simplicity.

\begin{proposition}
    Let $c(t) \in \mathcal{C}$ have hill interval $I_{\nu}$, and let $\{\T_n\}_{n\in \mathbb{N}}\subset \mathcal{I}$ be a sequence of time intervals exhausting $\R$. Then, $\mathrm{Cost}_t(c,\T_n)$ and $\mathrm{Cost}_y(c,\T_n)$ converge as $n \to \infty$ if and only if either $\nu \in \Lambda_{(a,b)}$ or $G_{\nu}(x) \equiv (1,0)$.
\end{proposition}

\begin{proof}
    Let $c(t) \in \mathcal{C}$ with momentum $\nu$. If $c(t)$ is a periodic geodesic, then Remark \ref{rem:one-form-close} and the inequality $0 \leq 1- G_1(x)$ imply that the cost accumulated over one period is positive. Since the intervals \(T_n\) exhaust \(\mathbb R\), there exists a subsequence \(\{T_{n_j}\}_{j\in\mathbb N}\) such that $[-jL(I_\nu), jL(I_\nu)]\subset \T_{n_\ell}$. Therefore
    $$  0<2j \int_{I_\nu} (1- G_1(x;\nu))\phi_{\nu} \leq \mathrm{Cost}_t(c,\T_{n_j}).  $$
    Hence, $\mathrm{Cost}_t(c,\T_{n_j}) \to \infty$ as $j \to \infty$. Thus, a necessary condition for the convergence of \(\operatorname{Cost}^t(c,T_n)\) is that \(c(t)\) is not periodic. By the classification of the reduced dynamics, the remaining cases are those for which $G_\nu(x(t))$ converges as $t \to \pm \infty$. For the cost functions to converge, it is sufficient that $G_{\nu}(x(t)) \to (1,0)$ as $t \to \pm \infty$. In this case, $c(t)$ is a homoclinic geodesic with asymptotic velocity $(0,1,0)$, since among the geodesics whose reduced algebraic curves are singular, $\R^5_{(a,b)}$ only admits geodesics whose reduced algebraic curves are singular of type homoclinic.

    If $c(t)$ is a homoclinic geodesic, then the oriented images \(x(T_n)\) exhaust \(I_\nu\) as $n \to \infty$. It follows that the sequence $\mathrm{Cost}_y(c,\T_n)$ converges if and only if $(1-P_1(x))\phi_{\nu}$ is smooth. Thus, the singularity of  $\phi_{\nu}$ at the endpoint of the hill interval must be removed. Consequently, $(1-P_1(x))\phi_{\nu}$ must extend smoothly to the endpoint of the hill interval. Since $1-P_1(x) = x^2$, the singularity of $\phi_\nu$ is removed if and only if the singular endpoint is $x=0$. Therefore, either $\nu \in \Lambda_{(a,b)}$ or $G_{\nu}(x) \equiv (1,0)$.   
 \end{proof}

\subsubsection{The Period Map}\label{sss:per-map}

The period map encodes the asymptotic behavior of the \sR geodesics with momentum  $\nu \in \Lambda_{(a,b)}$. 
\begin{definition}\label{def:period-map}
    The period map $\Theta_{(a,b)}^{\pm}:\mathcal{A}^{\pm}_{(a,b)} \to  \R^3$ is given by
    $$\Theta^{\pm}_{(a,b)}(\tau,\eta) : = (\Theta^{\pm,1}_{(a,b)}(\tau,\eta),\Theta^{\pm,2}_{(a,b)}(\tau,\eta),\Theta^{\pm,3}_{(a,b)}(\tau,\eta)),$$
     where
      \begin{equation*}
     \begin{split}
          \Theta^{\pm,1}_{(a,b)}(\tau,\eta)  &:= 2\int_{I_\nu^{\pm}} G_2(x;\nu) \phi_{\nu} = 2\int_{I_\nu^{\pm}} \eta(bx+ax^2) \phi_{\nu},\\
         \Theta^{\pm,2}_{(a,b)}(\tau,\eta)  & := 2\int_{I_\nu^{\pm}} G_1(x;\nu)(1 -P_1(x)) \phi_{\nu} = 2\int_{I_\nu^{\pm}} x^2(1 -\tau x^2) \phi_{\nu},\\
         \Theta^{\pm,3}_{(a,b)}(\tau,\eta)   &:= 2\int_{I_\nu^{\pm}} G_2(x;\nu)P_2(x) \phi_{\nu} = 2\int_{I_\nu^{\pm}} \eta(bx+ax^2)^2\phi_{\nu}. \\
     \end{split}
     \end{equation*}
\end{definition}
A fundamental result is the following.
\begin{proposition}\label{prp:no-zero-image}
If $(a,b) \neq (0,0)$, then  $(0,0,0) \notin \Theta_{(a,b)}^{\pm}(\mathcal{A}^{\pm}_{(a,b)})$.
\end{proposition}
\begin{proof}
    By Definition \ref{def:period-map}, $\Theta_{(a,b)}^{\pm,3}(\tau,\eta) = 0$ if and only if $\eta = 0$. Moreover, $\eta = 0$ implies $\Theta_{(a,b)}^{\pm,1}(\tau,\eta) = 0$. Therefore, it suffices to show that  $\Theta_{(a,b)}^{\pm,2}(\tau,0)\neq 0$ for all $\tau>0$. Using a change of variables $x=\sqrt{\frac{2}{\tau}}u$ and integration by parts, we obtain:
\begin{equation}\label{eq:theta-2-neg}
    \begin{split}
        \Theta_{(a,b)}^{\pm,2}(\tau,0) & = \int_{0}^{\sqrt{\frac{2}{\tau}}} \frac{x^2(1-\tau x^2)}{\sqrt{1-(1-\tau x^2)^2}}dx  \\
        & = -\frac{1}{\sqrt{2}\tau^{\frac{2}{3}}} \Big(  \int_0^1 \sqrt{1- (1- 2u^2)^2} du \Big) = -\frac{\sqrt{2}}{3\tau^{\frac{3}{2}}} <0.
    \end{split}
\end{equation}
Hence, $(0,0,0) \notin \Theta_{(a,b)}^{\pm}(\mathcal{A}^{\pm}_{(a,b)})$.
\end{proof}

The period map admits the following interpretation. Consider a geodesic $c(t) \in \mathcal{C}$ with momentum $\nu(\tau,\eta) \in \Lambda_{(a,b)}^+$ and hill interval $I_{\nu}^+$. Let $\{\T_n\}_{n\in \mathbb{N}}\subset\mathcal{I}$ be a sequence of time intervals exhausting $\R$. Then,
$$ \lim_{n\to \infty} (\Delta y^2(c,\T_n),\mathrm{Cost}_y(c,\T_n) ,\Delta z^2(c,\T_n)) = \Theta_{(a,b)}^+(\tau,\eta).$$
An analogous relation holds for the interval $I_{\nu}^-$ and the period map $\Theta_{(a,b)}^-$. 
Having established this interpretation of the period map, we now turn to its symmetries.
\begin{proposition}\label{prp:per-map-sym}
    The period map has the following symmetries:
    \begin{itemize}
        \item ($\eta$-reflection) 
        \begin{equation*}
            \begin{split}
            \Theta^{\pm,2}_{(a,b)}(\tau,\eta) & = \Theta^{\pm,2}_{(a,b)}(\tau,-\eta), \\
          \Theta^{\pm,i}_{(a,b)}(\tau,\eta) & = -\Theta^{\pm,i}_{(a,b)}(\tau,-\eta),\;\; i=1,3.       
            \end{split}
        \end{equation*}
        \item (hill-interval reflection) 
        $$\Theta_{(a,b)}^{-}(\tau,\eta) = \Theta_{(a,-b)}^{+}(\tau,\eta).$$
    \end{itemize}
\end{proposition}

A vector-valued polynomial is called even if $G_\nu(x) = G_\nu(-x)$. Let $\mathcal{A}_{(a,b)}^{\mathrm{even}}$ denote the subset of parameters $(\tau,\eta) \in \mathcal{A}_{(a,b)}^+\cap \mathcal{A}_{(a,b)}^-$ for which $G_{\nu}(x)$ is even.

\begin{proposition}\label{prop:even}
 The subset $\mathcal{A}_{(a,b)}^{\mathrm{even}}$ is characterized as follows:

\textbf{\textit{Case 1.}} If $b=0$, then $\mathcal{A}_{(a,0)}^{\mathrm{even}} = \mathcal{A}_{(a,0)}^+\cap \mathcal{A}_{(a,0)}^-$.

\textbf{\textit{Case 2.}} If $b \neq 0$, then $(\tau,\eta) \in \mathcal{A}_{(a,b)}^{\mathrm{even}}$ if and only if $\eta = 0$.

Moreover, if $G_{\nu}(x)$ is even, then $\Theta_{(a,b)}^+(\tau,\eta) = \Theta_{(a,b)}^-(\tau,\eta)$. 
\end{proposition}
\begin{proof}
    The first part follows directly from the definition of $G_{\nu}(x)$. Note that if $G_{\nu}(x)$ is even, then the roots defined in Eq. \eqref{eq:root-nu} satisfy $x^+(\nu) = -x^-(\nu)$. The change of variables $ u =- x$ in the integrals defining the period map yields the desired result.
\end{proof}

In fact, a stronger symmetry property holds. Let $(\tau,\eta) \in \mathcal{A}_{(a,b)}^{\mathrm{even}}$. If $c(t) = (x(t),\textbf{y}(t),\textbf{z}(t))$ is a geodesic with momentum $\nu(\tau,\eta)$, then $\widetilde{c}(t) = (-x(t),\textbf{y}(t),\textbf{z}(t))$ is also a geodesic with momentum $\nu(\tau,\eta)$. Consequently, every geodesic $c(t)$ with momentum $\nu(\tau,\eta)$ that crosses $x= 0$ twice fails to minimize; see \cite[Lemma 3]{bravo2022geodesics}.

\begin{thm}\label{the:per-map-into}
    Consider the period map $\Theta_{(a,b)}^{\pm}:\mathcal{A}_{(a,b)}^\pm \to \R^3$ where $(a,b)$ belongs to one of the two families defined in Section \ref{subsubsec:cand-geo}: 
    
    \textit{(Family 1)} $(a,0)$, $a\in \R\setminus\{0\}$,

    \textit{(Family 2)} $ (0,b)$, $b\in \R\setminus\{0\}$.


    
    Then the period map $\Theta_{(a,b)}^{\pm}$ is one-to-one for both families. 
\end{thm}

\subsubsection{A Sequence of Geodesics}\label{sss:sec-geo}

One of the main tools of the sequence method is a classical result in metric space theory.
\begin{proposition} [{\cite[Proposition 2.26]{bravododdoligeltype}}]\label{prop:min-set}
Let $K\subset M$ be a compact subset of a \sR manifold and let $\mathcal{T}$ be a compact time interval. Consider the following set
$$ \mathrm{Min}(K,\mathcal{T}) := \{ c:\R \to  M : \; c(\mathcal{T}) \subset K \; \text{and} \; c\;\text{is minimizing on}\;\mathcal{T} \} .   $$
Then, $\mathrm{Min}(K,\mathcal{T})$ is sequentially compact in the uniform norm.    
\end{proposition}
 Proposition \ref{prop:min-set} follows directly from the Arzel\`{a}--Ascoli theorem.

The following lemma is the last essential tool for the sequence method.
\begin{lemma}[{\cite[Lemma 2.27]{bravododdoligeltype}}]\label{lem:iso-metr}
Let $c_1,c_2:\R\to\R^{n+2}_{F}$ be two geodesics such that $c_1 \in \mathrm{Min}(K,\mathcal{T})$ and $\T' \subset\T$. If $\Phi(x,\textbf{y},\textbf{z})$ is an isometry such that $c_2(\mathcal{T}') \subset \Phi(c_1(\mathcal{T}))$, then $c_2 \in \mathrm{Min}(\Phi(K),\mathcal{T}')$. 
\end{lemma}

The following definition introduces a compactness condition for the $x$-component of a sequence of geodesics.
\begin{definition}\label{def:geo-comp-def}
Let $\Com(r,x_0,C)$ be the set of pairs $(c,\mathcal{T})$ satisfying the following conditions:
\begin{enumerate}
\item $c(t)$ is minimizing on $\mathcal{T}$.

\item $\|\mathrm{Cost}(c,\mathcal{T})\| \leq C$.

\item The endpoint values satisfy $x(t)\in B(r,x_0)$ for all $t \in \partial \mathcal{T}$, where $B(r,x_0)$ is the closed ball in $\R$ of radius $r$ centered at $x_0$, and $\partial \mathcal{T}$ is the boundary of $\mathcal{T}$.
\end{enumerate}

We say that the region $B(r,x_0) \times \R^4$ is geodesically compact if  for every $C>0$ and every sequence $\{(c_n,\mathcal{T}_n)\}_{n\in \mathbb{N}}$ such that $ (c_n,\mathcal{T}_n) \in \Com(r,x_0,C)$ for every $n \in \mathbb{N}$ and such that   $\{\mathcal{T}_n\}_{n\in \mathbb{N}}$ exhausts $\R$,
there exists a compact subset $K_{x}\subset \R$ such that $x_n(\mathcal{T}_n) \subset K_{x}$ for all $n$. Furthermore, we say that a magnetic space is geodesically compact if, for every $r>0$, the region $B(r,x_0) \times \R^4$ is geodesically compact.
\end{definition}

\begin{lemma}\label{lem:geo-com-ab}
The magnetic space $\R^5_{(a,b)}$ is geodesically compact for all $(a,b) \in \R^2\setminus\{(0,0)\}$. 
\end{lemma}
\begin{proof}
    By translation invariance, we may assume without loss of generality that $x_0 = 0$, and hence consider $B(r,0)$ for arbitrary $r>0$. We proceed by contraposition. Consider a sequence $\{(c_n,\mathcal{T}_n)\}_{n\in \mathbb{N}}$ such that $(c_n,\mathcal{T}_n) \in \Com(r,0,C)$ for all $n$. Assume that $\{x_n(\T_n)\}_{n\in \mathbb{N}}$ is unbounded. We show that $\{\mathrm{Cost}(c_n,\T_n)\}_{n\in \mathbb{N}}$ is also unbounded. Let $h_n: [0,1] \to x_n(\T_n)$ be the unique orientation-preserving affine bijection. Consider the vector-valued polynomial $\widehat{G}_{n}(s) = G_{n}(h_n(s))$ where the geodesic $c_n(t)$ has momentum $\nu_n$. It follows that $\widehat{G}_{n}:[0,1] \to \R^2$ is a vector-valued polynomial satisfying $\|\widehat{G}_{n}(s)\|_{\R^2} \leq 1$ for all $s \in [0,1]$. Since the functions \(\widehat G_n\) belong to a fixed finite-dimensional space of polynomial functions and are uniformly bounded on \([0,1]\), after passing to a subsequence, they converge uniformly to a \(\mathbb R^n\)-valued polynomial \(\widehat G:[0,1]\to\mathbb R^2\). We consider the cases $\widehat{G} \not\equiv (1,0)$ and $\widehat{G} \equiv  (1,0)$. For simplicity,  we relabel $n_\ell$ as $n$. 

    Case $\widehat{G} \not\equiv (1,0)$. Consider the change of variables $h_n(s) =x$. It follows that
    $$ Cost_{t}(c_n,\T_n) =\int_{0}^1(1-(1,0)\cdot \widehat{G}_{n}(s))h_n^*(\phi_n),$$
    since $h_n$ is afffine, $h_n^*(\phi_n)$ contains the multiplicative factor $|x_n(\T_n)|$, the length of the interval $x_n(\T_n)$; that is $h^*_n(\phi_n) = |x_n(\T_n)| \psi_n(s)ds$. The condition $\widehat{G}(s) \not\equiv (1,0)$ and Fatou's lemma imply
    \begin{equation*}
        \begin{split}
         0 &< |x_n(\T_n)|\int_0^1 (1-(1,0)\cdot \widehat{G}(s)) \psi_n(s)ds\\
         &\leq \liminf |x_n(\T_n)|\int_0^1 (1-(1,0)\cdot \widehat{G}_{n}(s)) \psi_n(s)ds.    
        \end{split}
    \end{equation*}
    Therefore, $|x_n(\T_n)| \to \infty$ implies $Cost_{t}(c_n,\T_n) \to \infty$ as $n \to \infty$.

    In the case $\widehat{G} \equiv(1,0)$, the same argument applied to $\mathrm{Cost}_{y}$ yields $\mathrm{Cost}_{y}(c_n,\T_n) \to \infty$ as $n \to \infty$.
\end{proof}


 \begin{definition}
 We say a sequence of geodesics $\{c_n(t)\}_{n \in \mathbb{N}}$ is strictly normal if each $c_n(t)$ is normal and every uniformly convergent subsequence has a normal limit.
 \end{definition}

\section{The Sequences Method}\label{sec:sec-method}
This section presents the sequence method. Section \ref{subsec:set-up} defines the sequence of sub-Riemannian geodesics employed in our proof, states Theorem \ref{thm:main-2}, applies the sequence method to prove it, and concludes with the proof of Theorem \ref{the:main}. Section \ref{subsub:per-map-II} establishes that the period map is one-to-one for the families of geodesics considered in Theorem \ref{the:main}.

\subsection{The Sequence Method}\label{subsec:set-up}
Let $c_h(t)$ be the candidate geodesic with momentum $\nu(1,1)$, defined in Eq. \eqref{eq:nu-par}, and write
$$ c_h(t) = (x_h(t),y^1_h(t),y_h^2(t),z^1_h(t),z^2_h(t)) .$$ 
Our goal is to show that $c_h(t)$ is minimizing on the interval $[-T,T]$ for an arbitrary $T$. The strategy is as follows: for all $n\in\mathbb{N}$, let $c_n:[0,T_n] \to \R^5_{(a,b)}$ be a length-minimizing geodesic connecting $c_h(-n)$ to $c_h(n)$; see Fig. \ref{fig:sequence-geo}. We then construct a reparametrization $\widetilde{c}_n(t)$ of $c_n(t)$ and a set $\mathrm{Min}(K,\T)$ such that $\widetilde{c}_n\in \mathrm{Min}(K,\T)$. We then extract a subsequence that converges to a geodesic $c_{\infty}(t)$. This limiting geodesic $c_{\infty}(t)$ satisfies:
$$c_h([-T,T]) \subset \Phi_{(\mathbf{y}_0,\mathbf{z}_0)}(c_\infty(\T))$$
for some isometry  $\Phi_{(\mathbf{y}_0,\mathbf{z}_0)}:\R^5_{(a,b)} \to \R^5_{(a,b)}$ defined by Eq. \eqref{eq:tran-isom}. By Lemma \ref{lem:iso-metr}, $c_h(t)$ is therefore minimizing  on $[-T,T]$. Since $T$ is arbitrary, it follows that $c_h$ is a metric line.

\begin{figure}[h!]
            \includegraphics[width=0.38\linewidth]{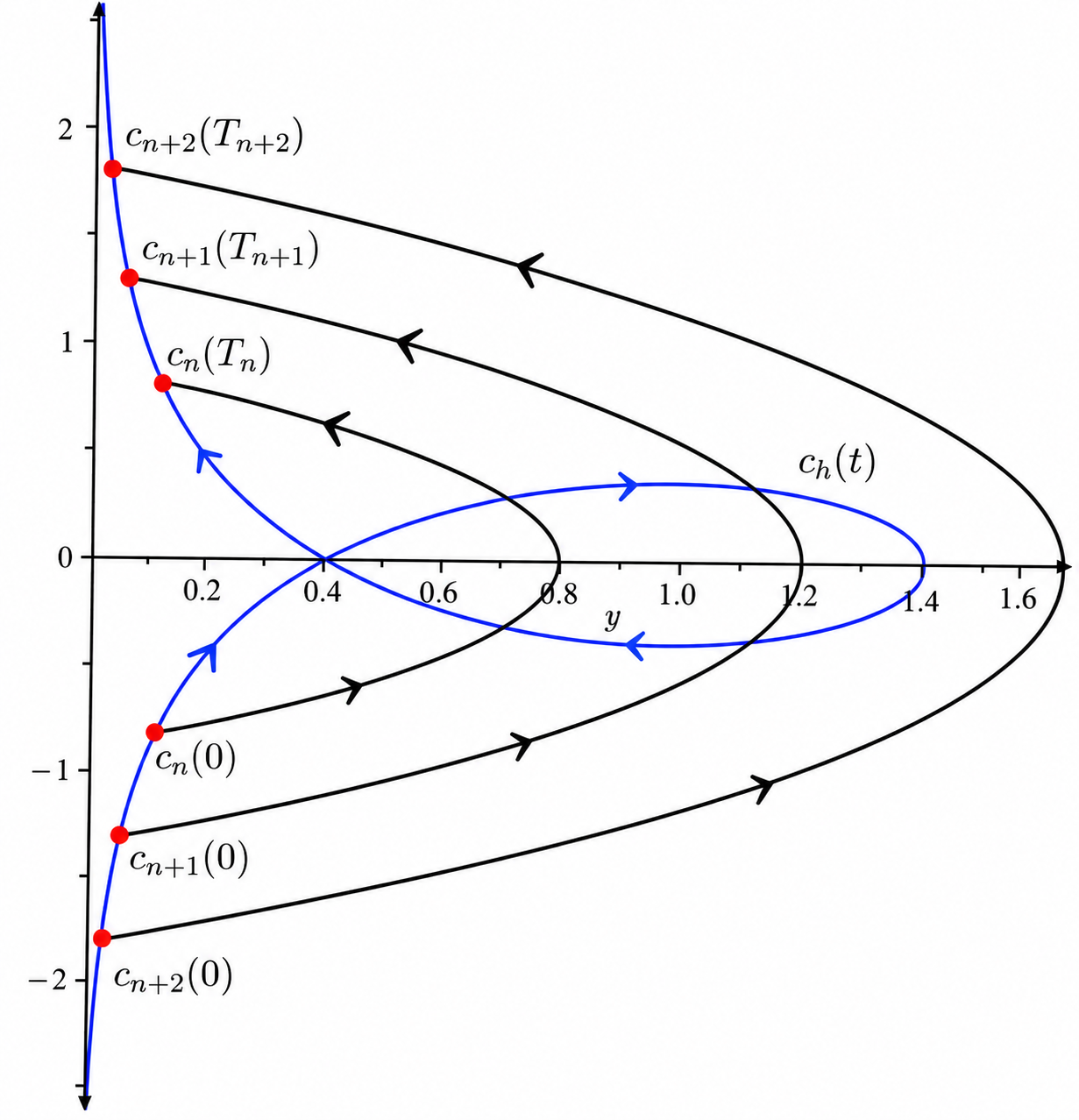}
        \caption{The panel illustrates the geometric construction of the sequence method. The candidate geodesic \(c_h(t)\) is shown in blue, and the geodesics in the minimizing sequence joining the red points on \(c_h(t)\) are shown in black.}
        \label{fig:sequence-geo}
    \end{figure}

By translating the time parameter and the $y$- and $z$-coordinates, we may assume without loss of generality that
 $$c_h(0) = (x^+(\nu(1,1)),0,0,0,0).$$ 
Since \(x_h(0)=x^+(\nu(1,1))\) is a turning point, the time-reversibility of the Hamiltonian system for \(H_\nu\), defined in Eq. \eqref{eq:red-ham}, implies that $x_h(t) = x_h(-t)$. Thus, $\Delta x(c_h,[-n,n]) :=x_h(n)-x_h(-n) =0$ for all $n\in \mathbb{N}$. 
With this choice of initial point, there exists a time $T_h^*>0$ such that
\begin{equation}\label{eq:new-T}
  y^1_h(t) > 0\;\;\text{and}\;\; y^1_h(-t) < 0\;\;\text{for all}\;\;t > T^*_h .  
\end{equation}
Indeed, by construction, 
\[y_h^1(t)\to\infty\quad\text{as }t\to\infty,\qquad \text{and} \qquad y_h^1(t)\to-\infty\quad\text{as }t\to-\infty. \]


\subsubsection{Construction of the Sequence.}
 Fix $T>T^*_h$. For each integer $n>T$, consider the points $c_h(-n)$ and $c_h(n)$, and let
$$c_n(t) = (x_n(t),y^1_n(t),y^2_n(t),z^1_n(t),z^2_n(t))$$
be a minimizing geodesic satisfying the endpoint and length conditions
\begin{equation}\label{eq:end-cond}
    c_n(0) = c_h(-n),\;\;c_n(T_n) = c_h(n),\;\;\text{and}\;\;T_n \leq 2n.
\end{equation}
We call the first two relations the endpoint conditions and the inequality the length condition. By the endpoint conditions and the asymptotic behavior of $c_h$, we obtain
\begin{equation}\label{eq:asy-cond}
\begin{split}
        x_n(0) \to 0,\;\;y^1_n(0) \to -\infty,\;\; z^1_n(0) \to -\infty \;&\; \text{as} \;\; n\to \infty, \\
         x_n(T_n) \to 0,\;\;y^1_n(T_n) \to \infty,\;\;z^1_n(T_n) \to \infty \;&\; \text{as} \;\; n\to \infty .
\end{split}
\end{equation}
Since the quantities $\Delta y^2$, $\mathrm{Cost}_y$, and  $\Delta z^2$ depend only on the endpoints, the endpoint conditions in Eq. \eqref{eq:end-cond} and the interpretation of the period map yield the following asymptotic period conditions:
\begin{equation}\label{eq:asy-per-con}
    \begin{split}
        \lim_{n\to \infty} \Delta y^2(c_n,[0,T_n]) &= \Theta_{(a,b)}^{+,1}(1,1),\\
        \lim_{n\to \infty} \mathrm{Cost}_y(c_n,[0,T_n]) & = \Theta_{(a,b)}^{+,2}(1,1), \\
        \lim_{n\to \infty} \Delta z^2(c_n,[0,T_n])&= \Theta_{(a,b)}^{+,3}(1,1). 
    \end{split}
\end{equation}
 Equation \eqref{eq:asy-per-con} establishes only the convergence of the quantities appearing there; it does not yet imply that the sequence of geodesics \(\{c_n\}\) converges to a specific geodesic.

Our second main result is the following.

\begin{theorem}\label{thm:main-2}
  Let $c_h$ be a homoclinic geodesic with momentum $\nu \in \Lambda_{(a,b)}$. Assume that the following conditions hold:
  
  \textbf{\textit{(1)}} There exists $C_\Theta^{\mathrm{Max}}>0$ such that $\|\mathrm{Cost}(c_h,\T)\|_{\infty} \leq C_\Theta^{\mathrm{Max}}$ for every $\T\subset \mathcal{I}$.

  \textbf{\textit{(2)}} The region $B(r,0)\times \R^4$ is geodesically compact for some $r>0$. 

  \textbf{\textit{(3)}} The sequence of geodesics $\{c_n(t)\}_{n>T}$ defined by Eq. \eqref{eq:end-cond}  is strictly  normal.

    \textbf{\textit{(4)}} The period maps $\Theta^\pm_{(a,b)}: \mathcal{A}^\pm_{(a,b)} \to  \R^3$ are one-to-one and satisfy
  $$\Theta^+_{(a,b)}(\mathcal{A}^+_{(a,b)}\setminus\mathcal{A}_{(a,b)}^{\mathrm{even}}) \cap \Theta^-_{(a,b)}(\mathcal{A}^-_{(a,b)}\setminus\mathcal{A}_{(a,b)}^{\mathrm{even}}) = \varnothing. $$

  Then $c_h$ is a metric line in $\R^5_{(a,b)}$. Consequently, if $\gamma_h \in \PJ$ is a horizontal lift of $c_h$, then $\gamma_h$ is a metric line in $\PJ$.
\end{theorem}

We now apply the sequence method to prove Theorem \ref{thm:main-2}.

\subsubsection{A New Initial Point}

By construction, the sequence of initial points  $\{c_n(0)\}$ is unbounded. The first step of the sequence method is to find a reparametrization $\widetilde{c}_n$ such that $\widetilde{c}_n(0) \in K_0$ for all $n> T$, where $K_0$ is a compact subset of $\R^5_{(a,b)}$.

Since \(y_n^1(0)<0<y_n^1(T_n)\), the intermediate value theorem implies that there exists \(t_n^*\in(0,T_n)\) such that \(y_n^1(t_n^*)=0\) for every $n >T$. Since $c_n$ is minimizing, Proposition \ref{prp:cost-short} and condition \textbf{\textit{(1)}} of Theorem \ref{thm:main-2} give a uniform bound on $\|\mathrm{Cost}(c_h,\T)\|_{\infty}$. Moreover, by Eq. \eqref{eq:asy-cond}, the endpoint values $x_n(0)$ and $x_n(T_n)$ lie in a fizex ball $B(r,0)$ for all sufficient large $n$. Hence, the sequence $\{(c_n,[0,T_n])\}_{n>T}$ satisfies the hypotheses in the definition of geodesic compactness. Therefore, condition \textbf{\textit{(2)}} of Theorem \ref{thm:main-2} implies that there exists a compact subset $K_{x}\subset \R$ such that $x_n([0,T_n]) \subset K_{x}$ for every $n>T$. In particular, $x_n(t^*_n) \in K_{x}$. Therefore, it is enough to construct compact subsets $K_{y}^2$, $K_{z}^1$ and $K_{z}^2$ such that $y_n^2(t^*_n) \in K_{y}^2$, $z_n^1(t^*_n) \in K_{z}^1$ and $z_n^2(t^*_n) \in K_{z}^2$. The existence of such compact sets is a consequence of the following result.

\begin{proposition}\label{prp:seq-bound}
There exist $\widetilde{C}^2_y$, $\widetilde{C}^1_z$, and $\widetilde{C}^2_z \in (0,\infty)$ such that 
$$|\Delta y^2(c_n,[0,t_n^*])| \leq \widetilde{C}^2_y,\;\;|\mathrm{Cost}_y(c_n,[0,t_n^*])| \leq \widetilde{C}^1_z,\;\; |\Delta z^2(c_n,[0,t_n^*])| \leq \widetilde{C}^2_z,$$
for all $n > T$. 
\end{proposition}
We first use Proposition \ref{prp:seq-bound} to construct $K_0$ and then prove the proposition.

We define the compact set
$$K_0 := K_x \times \{0\}\times K_{y}^2 \times K_{z}^1 \times K_{z}^2, $$
where $K_{y}^2 := [-\widetilde{C}^2_y,\widetilde{C}^2_y]$ and $K_{z}^2 := [-\widetilde{C}^2_z,\widetilde{C}^2_z]$. The construction of $K_{z}^1$ is as follows. By the definition in Section \ref{sss:cost-fuc}, we have 
\begin{equation*}
    \begin{split}
        z_n^1(t^*_n) - z_n^1(0) & = \Delta z^1(c_n,[0,t^*_n]) = \Delta y^1(c_n,[0,t^*_n]) - \mathrm{Cost}_y(c_n,[0,t^*_n]), \\
        z_h^1(0) - z_h^1(-n) & = \Delta z^1(c_h,[-n,0]) = \Delta y^1(c_h,[-n,0]) - \mathrm{Cost}_y(c_h,[-n,0]). \\
    \end{split}
\end{equation*}
By construction, $\Delta y^1(c_h,[-n,0]) = \Delta y^1(c_n,[0,t^*_n])$, $z_h^1(0) = 0$, and $ z_n^1(0) = z_h^1(-n)$ for all $n>T$. It follows that
\begin{equation}
\begin{split}
    |z_n^1(t^*_n)| & = |\Delta z^1(c_n,[0,t^*_n]) - \Delta z^1(c_h,[-n,0])| \\
    & \leq |\mathrm{Cost}_y(c_n,[0,t^*_n])| + |\mathrm{Cost}_y(c_h,[-n,0])|  \leq  \widetilde{C}^1_z + C_\Theta^{\mathrm{Max}} ,
\end{split}
\end{equation}
where $\widetilde{C}^1_z$ and $C_\Theta^{\mathrm{Max}}$ are given by  Proposition \ref{prp:seq-bound} and condition \textbf{\textit{(1)}} of Theorem \ref{thm:main-2}, respectively.
 Therefore, the compact set $K_{z}^1$ is defined by
$$ K_{z}^1 := [-C_\Theta^{\mathrm{Max}} - \widetilde{C}^1_z,C_\Theta^{\mathrm{Max}} + \widetilde{C}^1_z].$$
Thus, we obtain a compact set $K_0$ such that $c_n(t^*_n) \in K_0$ for all $n>T$. 

We consider the reparametrization $\widetilde{c}_n:\T_n\to \R^5_{(a,b)}$ given by $\widetilde{c}_n(t) : = c_n(t+t_n^*)$, where $\T_n := [-t_n^*,T_n-t^*_n]$. Therefore, $\{\widetilde{c}_n\}_{n>T}$ is a sequence of minimizing geodesics on the intervals $\T_n$ such that $\widetilde{c}_n(0) \in K_0$. Moreover, by construction, the sequences $\{\widetilde{c}_n(-t^*_n)\}_{n>T}$ and $\{\widetilde{c}_n(T_n-t^*_n)\}_{n>T}$ escape every compact set. Since $\widetilde{c}_n$ is minimizing and parametrized by arc length, we obtain
$$ dist_{\R^5_{(a,b)}}(\widetilde{c}_n(0), \widetilde{c}_n(-t^*_n))=t^*_n,\;\; dist_{\R^5_{(a,b)}}(\widetilde{c}_n(0), \widetilde{c}_n(T_n-t^*_n)) = T_n-t^*_n.$$
Hence 
$$ t^*_n \to\infty,\;\;T_n-t^*_n \to \infty\;\; \text{as}\;\;n\to \infty.$$
It follows that the intervals $\T_n$ exhaust $\R$ as $n \to \infty$.  After passing to a subsequence and relabeling, we may assume that $\T_n \subset \T_{n+1}$ for every $n>T$. 

\begin{proof}[Proof of Proposition \ref{prp:seq-bound}]
We first establish the bound $\widetilde{C}^2_y$. By condition \textbf{\textit{(3)}} of Theorem \ref{thm:main-2}, each geodesic $c_n$ is normal. Therefore, we can associate a vector-valued polynomial $G_{n}(x)$ with $c_n(t)$. Moreover, let $h_n: [0,1] \to x_n([0,t_n^*])\subset K_x$ denote the unique orientation-preserving affine map sending $[0,1]$ onto $x_n([0,t_n^*])$. It follows that we can rewrite $\Delta y^2(c_n,[0,t_n^*])$ as follows:
\begin{equation*}
    \begin{split}
      \Delta y^2(c_n,[0,t_n^*]) &  = \int_{h^{-1}_n(x_n([0,t_n^*]))} h_n^*(G_2(x;\nu_n)\phi_{\nu_n}) \\
      & = u_n \int_{h^{-1}_n(x_n([0,t_n^*]))}\frac{G_2(h_n(s);\nu_n) }{\sqrt{1-\|\hat{G}_{n}(s)\|^2}} ds, 
    \end{split}
\end{equation*}
where $s \in [0,1]$, $u_n \in (0,\infty)$ is the length of the interval $x_n([0,t_n^*]) \subset \R$, and $\widehat{G}_{n} := G_{\nu_n}\circ h_n$ is a vector-valued polynomial defined on the interval $[0,1]$. In addition, since \(x_n([0,t_n^*])\subset K_x\), the sequence \(\{u_n\}\) is uniformly bounded. Moreover, the functions \(G_{\nu_n}\circ h_n\) are uniformly bounded on \([0,1]\).

Equation \eqref{eq:asy-per-con} implies that the sequence $\Delta y^2(c_n,[0,T_n])$ is bounded. Remark \ref{rem:one-form-close} then implies that the one-form $G_1(x;\nu_n)\phi_{\nu_n}$ is smooth on the interval $x_n([0,t_n^*])$. Consequently, its pullback $h_n^*(G_1(x;\nu_n)\phi_{\nu_n})$ is smooth. Hence its integral over \(h_n^{-1}([x_n(0),x])\) is well-defined and finite for every \(n\) and every \(x\in x_n([0,t_n^*])\).

We now obtain a uniform bound. Since \(G_{\nu_n}\) is a \(\mathbb R^n\)-valued polynomial of degree at most two and \(h_n\) is affine, the functions $\widehat G_n$ 
belong to the finite-dimensional space \(\mathcal P_2([0,1],\mathbb R^2)\) of \(\mathbb R^2\)-valued polynomials of degree at most two. Moreover, the functions \(\widehat G_n\) are uniformly bounded on \([0,1]\), and the sequence \(\{u_n\}\) is uniformly bounded. The sequence \(\{(\widehat G_n,u_n)\}\) is bounded in the finite-dimensional space \(\mathcal P_2([0,1],\mathbb R^2)\times\mathbb R\). Hence, its closure $\mathcal{K}:= \overline{\{(\widehat G_n,u_n):n>T\}}$  is compact. We then define $\widetilde{C}^2_y$ as follows:
\begin{equation*}
    \widetilde{C}^2_y:= \sup_{n> T}  \biggl\{ \sup_{x \in x_n([0,t_n^*])} \biggl\{ \int_{h^{-1}_n([x_n(0),x])} |G_1(h_n(s);\nu_n)| h_n^*\phi_{\nu_n} \biggr\} \biggr\}.
\end{equation*}
Similarly, we obtain the bounds $\widetilde{C}^1_z$ and $\widetilde{C}^2_z$ for $\mathrm{Cost}_y(c_n,[0,t_n^*])$ and $\Delta z^2(c_n,[0,t_n^*])$, respectively. 
\end{proof}

\subsubsection{Construction of $\mathrm{Min}(K,\T)$} Consider $N\in \mathbb{N}$ such that $N>T$. We will construct a compact set $K_N$ such that $\widetilde{c}_n \in  \mathrm{Min}(K_N,\T_N$) for all $n>N$. By the preceding construction, we have that $\widetilde{c}_n(t)$ is minimizing on $\T_n$, and $\T_N \subset \T_n$, so $\widetilde{c}_n(t)$ is minimizing on $\T_N$. Moreover, $\|\dot{\widetilde{c}}_n(t)\| =1 $ implies that  
$$|\Delta x(\widetilde{c}_n,\T_N)| \leq T_N,\;\; |\Delta y^1(\widetilde{c}_n,\T_N)| \leq T_N,\;\;|\Delta y^2(\widetilde{c}_n,\T_N)| \leq T_N.$$
Let us bound $\Delta z^1(\widetilde{c}_n,\T_n)$ and $\Delta z^2(\widetilde{c}_n,\T_n)$. Using Eq. \eqref{eq:y-z-ode}, we have:
\begin{equation*}
|\Delta z_i(\widetilde{c}_n,\T_N) |   \leq \int_{\T_N} |P_i(x)|dt \leq   T_N \max_{x\in K_x+[-T_N,T_N]} | P_i(x)|=: C_z^i, \;\text{for }\;i=1,2.
\end{equation*}
Therefore, define the compact set $K_1\subset \R^5_{(a,b)}$ by
$$K_1 := [-T_N,T_N] \times [-T_N,T_N]\times [-T_N,T_N] \times [-C_z^1,C_z^1] \times [-C_z^2,C_z^2].$$
Then 
$$\widetilde{c}_n(\T_N) \subset K_N:= \{q_0+q_1:q_0\in K_0, q_1\in K_1\}.$$
Since $K_0$ and $K_1$ are compact, $K_N$ is compact. 

We conclude that $\widetilde{c}_n \in \mathrm{Min}(K_N,\T_N)$. 

\subsubsection{Proof of Theorem \ref{thm:main-2}} We are ready to prove the second main theorem of this paper. 

\begin{proof}[Proof of Theorem \ref{thm:main-2}]
     We have constructed a sequence of geodesics  $\widetilde{c}_n \in \mathrm{Min}(K_N,\T_N)$. By Proposition \ref{prop:min-set}, there exists a subsequence $\{\widetilde{c}_{n_\ell}\}$ converging to $\widetilde{c}(t)$. Condition \textbf{\textit{(3)}} implies that $\widetilde{c}$ is a normal geodesic, so $\widetilde{c}$ is associated with a vector-valued polynomial $G_{\nu}(x)$. Moreover, Eqs. \eqref{eq:asy-cond} and \eqref{eq:asy-per-con} imply that $\nu \in \Lambda_{(a,b)}$, and
$$\Theta^+_{(a,b)}(1,1) = \lim_{n_\ell \to \infty} (\Delta y^2(c_{n_\ell},\T_{n_\ell}), \mathrm{Cost}_y(c_{n_\ell},\T_{n_\ell}), \Delta z^2(c_{n_\ell},\T_{n_\ell})) .$$ 
Therefore, condition \textbf{\textit{(4)}} implies that $G_{\nu}(x) = G_{\nu(1,1)}(x) = P(x)$ and the hill interval of $\widetilde{c}(t)$ is $I_{\nu(1,1)}^+$. Any two geodesics associated with the vector-valued polynomial $P(x)$ and having the same hill interval are related by an isometry $\Phi_{(\mathbf{y}_0,\mathbf{z}_0)}(x,\mathbf{y},\mathbf{z})$. Since $c_h$ and $\tilde{c}$ are related by the isometry $\Phi_{(\mathbf{y}_0,\mathbf{z}_0)}(x,\mathbf{y},\mathbf{z})$, and the intervals $\T_N$ exhaust $\R$, we may  chose \(N\) sufficiently large so that $c_h([-T,T]) \subset \Phi_{(\mathbf{y}_0,\mathbf{z}_0)}(\widetilde{c}(\T))$. Lemma \ref{lem:iso-metr} implies that $c_h$ is minimizing on $[-T,T]$, and $T$ is arbitrary. Hence, $c_h$ is a metric line in $\R^5_{(a,b)}$.  
\end{proof}

\subsubsection{Proof of Theorem \ref{thm:main-3}} We now prove Theorem \ref{thm:main-3}, which, together with the discussion in Section \ref{subsubsec:cand-geo}, implies Theorem \ref{the:main}. 

\begin{proof}[Proof of Theorem \ref{thm:main-3}] It is enough to verify that the conditions of Theorem \ref{thm:main-2} hold for each family of geodesics. We first verify condition \textbf{\textit{(1)}}, which holds in general; that is, $\mathrm{Cost}(c_h,[-n,n])$ is uniformly bounded.  Since $|G_1(x)| \leq 1$ on $I^+_\nu$, Proposition \ref{pro:cost-G} yields
\begin{equation*}\label{eq:bnd-cost-fuc}
   \Theta_{(a,b)}^{\mathrm{Max}} := \Theta^{+,1}_{(a,b)}(1,1) \geq \mathrm{Cost}_t(c_h,[-n,n])  \geq  |\mathrm{Cost}_y(c_h,[-n,n])| .  
\end{equation*}

Lemma \ref{lem:geo-com-ab} implies condition \textbf{\textit{(2)}}. 

Proposition \ref{prp:abn-mag-spa} implies condition \textbf{\textit{(3)}}. Moreover, $\{c_n\}$ cannot converge to a line geodesic with asymptotic velocity $(0,1,0)$.  If $\{c_n\}$ converged to such a line geodesic, then the corresponding quantities in Eq. \eqref{eq:asy-per-con} would converge to \((0,0,0)\), contradicting Proposition \ref{prp:no-zero-image}.

We prove condition \textbf{\textit{(4)}} separately for the two families.

\emph{Case $(a,0)$ with $a \in \R\setminus\{0\}$.}  The required disjointness follows from the fact that \(G_\nu\) is even for every $\nu \in \mathcal{A}_{(a,0)}^{\pm}$. The injectivity of the period maps will follow from Theorem \ref{the:per-map-into}.

\emph{Case $(0,b)$ with $b \in \R\setminus\{0\}$.} To verify condition \textbf{\textit{(4)}}, observe that  $ G_{1}(x;\nu) \geq 0$ if $x \in I_{\nu}^+$, whereas $G_{1}(x;\nu)\leq 0$ if $x \in I_{\nu}^-$. It follows that if $(\tau,\eta) \in \mathcal{A}^+_{(0,b)} \setminus\mathcal{A}_{(0,b)}^{\mathrm{even}}$ then $0 < \Theta^{+,1}_{(0,b)}(\tau,\eta)$, and $ \Theta^{+,1}_{(0,b)}(\tau,\eta) = -\Theta^{-,1}_{(0,b)}(\tau,\eta)$.
This sign property establishes the required disjointness of the two images. Thus, Theorem \ref{thm:main-3}, and hence Theorem \ref{the:main}, follow once Theorem \ref{the:per-map-into} is established. The next section proves it.
\end{proof}

\subsection{The Period Map II}\label{subsub:per-map-II}

The goal of this section is to prove Theorem \ref{the:per-map-into}.  We first analyze the cases \(\eta>0\) and \(\eta=0\), and then extend the result to \(\eta<0\). For \(\eta>0\), we apply the Hadamard global diffeomorphism theorem. We begin by stating it formally.

\begin{thm}[Hadamard global diffeomorphism theorem \cite{krantz2002implicit}]
    Let $M$ and $N$ be connected smooth manifolds of the same dimension, and suppose that $N$ is simply connected. If $F:M \to N$ is a proper smooth map such that $dF_{p}:T_pM\to T_{F(p)}N$ is an isomorphism for every $p\in M$, then $F$ is a diffeomorphism.    
\end{thm}

We cannot apply the Hadamard global diffeomorphism theorem directly to the period map $\Theta_{(a,b)}^+(\tau,\eta)$, since its domain and target have different dimensions. To address this obstacle, we introduce the auxiliary map $f:\R^3 \to \R^2$ given by
\begin{equation*}
     f(\Theta^{1},\Theta^{2},\Theta^{3}) = (\Theta^{2},\Theta^{3}).
\end{equation*}
We now define the sets
$$M_{(a,b)} := \Int \mathcal{A}_{(a,b)}^{+}\cap\{(\tau,\eta):\eta>0\} \;\;\text{and}\;\; N_{(a,b)}:= (f\circ\Theta^+_{(a,b)})(M_{(a,b)}).$$
Therefore, we  consider the function
\begin{equation}\label{eq:F-def}
    F_{(a,b)}:M_{(a,b)} \to N_{(a,b)},\;\; F_{(a,b)}:= f \circ \Theta_{(a,b)}^{+}.
\end{equation}
If $F_{(a,b)}$ is one-to-one, then $\Theta_{(a,b)}^{+}$ is one-to-one. 

For the remainder of this section, we restrict our attention to the two cases \((a,0)\), \(a\neq0\), and \((0,b)\), \(b\neq0\). We first present the explicit formulas for \(F_{(a,b)}\) in these two cases.
    \begin{equation}\label{eq:F-b-eq-0}
    \begin{split}
        F_{(a,0)}(\tau,\eta) & = \left(\frac{\sqrt{2\tau}(3a^2\eta^2-\tau^2)}{3(a^2\eta^2+\tau^2)^2},\frac{2a^2\eta(2\tau)^{\frac{3}{2}}}{3(a^2\eta^2 + \tau^2)^2}\right), \\
         F_{(0,b)}(\tau,\eta) & = \left(\frac{\sqrt{2\tau-b^2\eta^2}(2b^2\eta^2 -\tau)}{3\tau^3},\frac{b^2\eta\sqrt{2\tau-b^2\eta^2}}{\tau^2}\right).\\
    \end{split}
    \end{equation}

\begin{lemma}\label{lem:jac-F}
    Consider the map $ F_{(a,b)}:M_{(a,b)} \to N_{(a,b)}$. Then $dF_{(a,b)}|_{(\tau,\eta)}$ is an isomorphism for every $(\tau,\eta) \in M_{(a,b)}$ in either of the following cases:
    \begin{itemize}
        \item $b=0$ and $a \neq 0$;
        \item $a=0$ and $b \neq 0$.
    \end{itemize}
\end{lemma}

\begin{proof}
    \emph{Case $b=0$ and $a \neq 0$.} 
    Equation \eqref{eq:F-b-eq-0} implies 
    $$ \det dF_{(a,0)} = \frac{4a^2\eta^2\tau(3a^2+1)}{3(a^2\eta^2 + \tau^2)^4}. $$
    Since $(\tau,\eta) \in M_{(a,0)}$, we have $\tau>0 $ and $\eta>0$. Thus, $\det d F_{(a,0)} \neq 0$, and hence $d F_{(a,0)}|_{(\tau,\eta)}$ is an isomorphism.

 \emph{Case  $a=0$ and $b \neq 0$}. 
    Equation \eqref{eq:F-b-eq-0} implies
    $$\det d F_{(0,b)} =\ \frac{b^2}{\tau^4}.$$ 
    Since $b \neq 0$ and $(\tau,\eta) \in M_{(0,b)}$, we have $\tau>0$. Thus $\det d F_{(0,b)} \neq 0$, and therefore $d F_{(0,b)}|_{(\tau,\eta)}$ is an isomorphism.
\end{proof}

\begin{lemma}\label{lem:prop-F}
    The map \(F_{(a,b)}:M_{(a,b)}\to N_{(a,b)}\) is proper in either of the following cases: $b=0$, and $a\neq 0$, or $a=0$, $b\neq 0$. 
\end{lemma}
Recall that a continuous map \(F:M\to N\) is proper if and only if every sequence in \(M\) that escapes every compact subset of \(M\) is mapped to a sequence that escapes every compact subset of \(N\).
\begin{proof}
Let \(\{(\tau_n,\eta_n)\}\) be a sequence that escapes every compact subset of \(M_{(a,b)}\). Suppose, for contradiction, that \(\{F_{(a,b)}(\tau_n,\eta_n)\}\) does not escape every compact subset of \(N_{(a,b)}\). Then there exists a compact set \(K\subset N_{(a,b)}\) containing infinitely many terms of the sequence. Passing to a subsequence, we may assume that
$$F_{(a,b)}(\tau_n,\eta_n) \to (L_1,L_2) \in K.$$
Since \(\Theta^{+,3}_{(a,b)}(\tau,\eta)>0\) for \((\tau,\eta)\in M_{(a,b)}\), we have $N_{(a,b)}\subset\mathbb R\times(0,\infty)$.
We show that \(L_2=0\), which contradicts \((L_1,L_2)\in N_{(a,b)}\).

\emph{Case $b=0$ and $a \neq 0$.}  In this case, $M_{(a,0)}=(0,\infty)^2$. Define $r=\frac{|a|\eta}{\tau}>0$, then Eq. \eqref{eq:F-b-eq-0} implies
\begin{equation}\label{eq:F-b-eq-0-1}
    F_{(a,0)}(\tau,\eta) =
\tau^{-3/2}
\left(
\frac{\sqrt{2}(3r^2-1)}{3(1+r^2)^2},
\frac{2^{5/2}|a|r}{3(1+r^2)^2}
\right).
\end{equation}
Since \(\{(\tau_n,\eta_n)\}\) escapes every compact subset of  \(M_{(a,0)}\), after passing to a further subsequence, at least one of the following holds:
$$\tau_n \to 0,\;\;\;\tau_n \to \infty, \;\;\;\eta_n\to 0 \;\;\text{or}\;\;\eta_n\to \infty.$$
Indeed, otherwise the sequence would eventually be contained in a compact rectangle $[\epsilon,R]^2 \subset M_{(a,0)}$. 

\emph{Case $\tau_n \to \infty$.} Equation \eqref{eq:F-b-eq-0-1} implies that $F_{(a,0)}(\tau_n, \eta_n) \to (0,0)$ as $n \to \infty$. 

\emph{Case $\eta_n \to 0$ and $\tau_n \to \tau_0 \in (0,\infty)$.} In this case $r_n \to 0$, and Eq. \eqref{eq:F-b-eq-0-1} implies that $F_{(a,0)}(\tau_n, \eta_n) \to (-\frac{\sqrt{2}}{3\tau_0^{\frac{3}{2}}},0)$ as $n \to \infty$.

\emph{Case $\eta_n \to \infty$ and $\tau_n >C$ for some $C>0$.} In this case $r_n \to \infty$, and Eq. \eqref{eq:F-b-eq-0-1} implies that $F_{(a,0)}(\tau_n, \eta_n) \to (0,0)$ as $n \to \infty$. If \(\tau_n\) is not bounded away from zero, we may pass to a further subsequence such that \(\tau_n\to0\), which is covered by the final case.

\emph{Case $\tau_n \to 0$.} Since \(\{F_{(a,0)}(\tau_n,\eta_n)\}\) converges, it is bounded. Equation \eqref{eq:F-b-eq-0-1} implies that $r_n \to \infty$. Moreover, boundedness of the first component gives $\frac{1}{\tau_n^{\frac{3}{2}}r_n^2} = O(1)$, and hence $r_n^{-1} = O(\tau^{\frac{3}{4}}_n)$. Therefore, the second component satisfies
$$ \frac{2^{\frac{5}{3}}|a|r_n}{3\tau_n^{\frac{3}{2}}(1+r_n^2)^2} = O(\frac{1}{\tau_n^{\frac{3}{2}}r_n^3}) =O(\tau_n^{\frac{3}{4}}) \to 0.$$
Hence \(L_2=0\).

 \emph{Case  $a=0$ and $b \neq 0$.} If $(\tau,\eta) \in M_{(0,b)}$ then $b^2\eta^2<2\tau$. Set $s:=\frac{b^2\eta^2}{2\tau} \in (0,1)$. Then Eq. \eqref{eq:F-b-eq-0} implies
\begin{equation}\label{eq:F-b-eq-0-2}
    F_{(0,b)}(\tau,\eta)
=
\left(
\frac{\sqrt{2}(4s-1)\sqrt{1-s}}{3\tau^{3/2}},
\frac{2|b|\sqrt{s(1-s)}}{\tau}
\right).
\end{equation}
The map $(\tau,\eta) \to (\tau,s)$ identifies $M_{(0,b)}$ with $(0,\infty)\times (0,1)$. 

Since \(\{(\tau_n,s_n)\}\) escapes every compact subset of \((0,\infty)\times(0,1)\), after passing to a further subsequence, at least one of the following holds:
$$\tau_n \to 0,\;\;\;\tau_n \to \infty, \;\;\;s_n\to 0 \;\;\text{or}\;\;s_n\to 1.$$
Indeed, otherwise \(\{(\tau_n,s_n)\}\) would eventually be contained in a compact rectangle $[\epsilon,R]\times[\epsilon,1-\epsilon]$.

\emph{Case $\tau_n \to \infty$.} Eq. \eqref{eq:F-b-eq-0-2} implies that $F_{(0,b)}(\tau_n, \eta_n) \to (0,0)$ as $n \to \infty$. 

\emph{Case $s_n \to 0$ and $\tau_n \to \tau_0\in (0,\infty)$.} Eq. \eqref{eq:F-b-eq-0-2} implies $F_{(0,b)}(\tau_n,\eta_n)
\to
(
-\frac{\sqrt{2}}{3\tau_0^{3/2}},
0
)$ as $n\to \infty$.

\emph{Case $s_n \to 1$ and $\tau_n \to \tau_0\in (0,\infty)$.} Eq. \eqref{eq:F-b-eq-0-2} implies $F_{(0,b)}(\tau_n,\eta_n)
\to
(0,0)$ as $n\to \infty$.

\emph{Case $\tau_n \to 0$.} Since \(\{F_{(0,b)}(\tau_n,\eta_n)\}\) converges, it is bounded. We claim that \(s_n\to1\). Otherwise, after passing to a subsequence, we may assume that $s_n \to s_0 \in [0,1)$. If \(s_0\neq1/4\), then $(4s_n-1) \sqrt{1-s_n} \to (4s_0-1) \sqrt{1-s_0} \neq 0$ and hence the first component in Eq. \eqref{eq:F-b-eq-0-2} is unbounded as \(\tau_n\to0\). If \(s_0=1/4\), then $ \sqrt{s_n(1-s_n)} \to \frac{\sqrt{3}}{4}>0$, so the second component in Eq. \eqref{eq:F-b-eq-0-2} is unbounded. Both cases contradict the boundedness of \(\{F_{(0,b)}(\tau_n,\eta_n)\}\). Therefore, $s_n \to 1$. Set $\epsilon_n = 1-s_n$. Since \(4s_n-1\to3\), the boundedness of the first component in Eq. \eqref{eq:F-b-eq-0-2} implies $\frac{\sqrt{\epsilon_n}}{\tau^{\frac{3}{2}}_n} = O(1)$. Hence, $\sqrt{\epsilon_n} = O(\tau^{\frac{3}{2}}_n)$. Therefore, the second component satisfies
$$ \frac{2|b|\sqrt{s_n(1-s_n)}}{\tau_n} = \frac{2|b|\sqrt{s_n\epsilon_n}}{\tau_n} = O(\tau_n^{\frac{1}{2}}) \to 0.$$
Hence \(L_2=0\).

Thus, in both cases \(b=0,\ a\neq0\) and \(a=0,\ b\neq0\), we obtain \(L_2=0\). This contradicts
$$(L_1,L_2) \in K \subset N_{(a,b)} \subset \R\times (0,\infty).$$
Therefore, \(\{F_{(a,b)}(\tau_n,\eta_n)\}\) escapes every compact subset of \(N_{(a,b)}\). Hence \(F_{(a,b)}\) is proper.
\end{proof}

\begin{lemma}\label{lem:hyp-Hard-the}
    Consider the map $F_{(a,b)}:M_{(a,b)} \to N_{(a,b)}$ defined in Eq. \eqref{eq:F-def}. Then:
    \begin{enumerate}
        \item $M_{(a,b)}$ is connected.
        \item $N_{(a,0)}$ is simply connected for every \(a\neq0\), and \(N_{(0,b)}\) is simply connected for every \(b\neq0\).
    \end{enumerate}

\end{lemma}
\begin{proof}
 For \textit{(1)}, by the definition of $\mathcal{A}^+_{(a,b)}$, the set \(M_{(a,b)}\) is convex and hence connected.

   For \textit{(2)}, we first prove that $N_{(a,0)}$ is simply connected. Consider the action of $\R^+:=(0,\infty)$ on $\mathcal{A}_{(a,0)}^{+}$ defined by 
   $$\lambda \cdot (\tau,\eta) := (\frac{\tau}{\lambda^2},\frac{\eta}{\lambda^2}).$$
   It follows from Eq. \eqref{eq:Acal-1}  that $\lambda \cdot (\tau,\eta) \in \mathcal{A}_{(a,0)}^{+}$ for all $\lambda>0$. This action on $\mathcal{A}_{(a,0)}^{+}$ induces an action on $N_{(a,0)}$, given by $\lambda \cdot (\Theta^{+,2},\Theta^{+,3}) = (\lambda^3\Theta^{+,2},\lambda^3 \Theta^{+,3})$.

   Now, let $\sigma:[0,1] \to N_{(a,0)}$ be an arbitrary loop. Since $(0,0) \notin N_{(a,0)}$, we have $\|\sigma(s)\| \neq 0$ for every $s \in [0,1]$. Define $\lambda(s) :=\|\sigma(s)\|^{-\frac{1}{3}}$. By the action definition, the curve $\hat{\sigma}(s):= \lambda(s) \cdot \sigma(s)$ satisfies $\|\hat{\sigma}(s)\| = 1 $ for every \(s\in[0,1]\). Moreover, since \(\Theta^{+,3}(\tau,\eta)>0\) for \(\eta>0\), we have
   $$\hat{\sigma}([0,1])\subset \{  (\cos(\psi),\sin(\psi)) : \psi \in (0,\pi) \}. $$ 
    Since \(\rho\lambda(s)+(1-\rho)>0\), the homotopy
   $$ H(\rho,s) =
        \big( \rho\lambda(s)+(1-\rho)\ \big)\cdot \sigma(s)$$
   is well defined and deforms \(\sigma\) into \(\widehat\sigma\). Since the upper semicircle is contractible, \(\widehat\sigma\) is contractible in \(N_{(a,0)}\), and hence so is \(\sigma\). Therefore, \(N_{(a,0)}\) is simply connected.

   The proof for the case $N_{(0,b)}$  proceeds similarly. Here, we consider the action of $\R^+$ on $\mathcal{A}^+_{(0,b)}$ given by $\lambda \cdot(\tau,\eta) = (\frac{\tau}{\lambda^2}, \frac{\eta}{\lambda})$, which induces the corresponding action on $N_{(0,b)}$ given by $\lambda \cdot (\Theta^{+,2},\Theta^{+,3}) = (\lambda^3\Theta^{+,2},\lambda^2 \Theta^{+,3})$. Indeed, let $\sigma:[0,1] \to N_{(0,b)}$ be an arbitrary loop. 
   For every \((\Theta^{+,2},\Theta^{+,3})\in N_{(0,b)}\), we have
   $\|\lambda \cdot (\Theta^{+,2},\Theta^{+,3})\| \to \infty$ as $\lambda \to \infty$, while $\lambda \cdot (\Theta^{+,2},\Theta^{+,3}) \to (0,0)$ as $\lambda\to 0$. Therefore, for each $s \in [0,1]$, there exists a $\lambda(s) \in \R^+$ such that $\|\lambda(s) \cdot \sigma(s)\| = 1$.
\end{proof}

The following result shows that, to prove $\Theta_{(a,b)}^{\pm}$ are one-to-one, it suffices to consider its restriction to \(\eta>0\).

\begin{lemma}\label{lem:per-map-nu-pos}
    Let $(a,b) \in \R^2\setminus\{(0,0)\}$. If the restriction of \(\Theta^+_{(a,b)}\) to \(A^+_{(a,b)}\cap\{\eta>0\}\) are one-to-one, then \(\Theta^\pm_{(a,b)}:A^\pm_{(a,b)}\to\mathbb R^3\) are one-to-one. 
\end{lemma}

\begin{proof}
We first extend the one-to-one property from $\eta>0$ to $\eta\geq 0$, and then use Proposition \ref{prp:per-map-sym} to treat $\eta<0$.


\emph{Case $\eta = 0$.} By Definition \ref{def:period-map}, $\Theta_{(a,b)}^{+,3}(\tau,\eta) = 0$ if and only if $\eta = 0$. Therefore, 
$$\Theta_{(a,b)}^{+}(\mathcal{A}_{(a,b)}^{+}\cap\{\eta>0\}) \cap \Theta_{(a,b)}^{+}(\mathcal{A}_{(a,b)}^{+}\cap\{\eta=0\}) = \varnothing.$$
Moreover, Eq. \eqref{eq:theta-2-neg} gives
\begin{equation*}
    \begin{split}
        \Theta_{(a,b)}^2(\tau,0) &  = -\frac{\sqrt{2}}{3\tau^{\frac{3}{2}}}.
    \end{split}
\end{equation*}
Since the right-hand side is strictly increasing for \(\tau>0\), the restriction of \(\Theta^+_{(a,b)}\) to \(A^+_{(a,b)}\cap\{\eta=0\}\) is injective.


\emph{Case $\eta<0$.} By construction, $\Theta_{(a,b)}^{+,3}(\tau,\eta) < 0$ for all $\eta<0$. Proposition \ref{prp:per-map-sym} gives
$$ \Theta_{(a,b)}^+(\tau,-\eta) = (-\Theta_{(a,b)}^{+,1}(\tau,\eta),\Theta_{(a,b)}^{+,2}(\tau,\eta),-\Theta_{(a,b)}^{+,3}(\tau,\eta)) .$$
Since $ \Theta_{(a,b)}^{+}$ is one-to-one on $\mathcal{A}_{(a,b)}^{+}\cap\{\eta>0\}$, this symmetry implies that it is also one-to-one on $\mathcal{A}_{(a,b)}^{+}\cap\{\eta<0\}$. 
Moreover, the sign of the third component implies that
$$ \Theta_{(a,b)}^+\big(\mathcal{A}_{(a,b)}^{+}\cap\{\eta>0\}\big)\cap \Theta_{(a,b)}^+\big( \mathcal{A}_{(a,b)}^{+}\cap\{\eta<0\} \big) = \varnothing.$$
Together with the case \(\eta=0\) considered above, this proves that $\Theta_{(a,b)}^+: \mathcal{A}_{(a,b)}^+ \to \R^3 $ is one-to-one. Finally, Proposition \ref{prp:per-map-sym} gives  $\Theta_{(a,b)}^-(\tau,\eta) = \Theta_{(a,-b)}^+(\tau,\eta)$. Hence $\Theta_{(a,b)}^-: \mathcal{A}_{(a,b)}^- \to \R^3$ is one-to-one.
\end{proof}

\subsubsection{Proof of Theorem \ref{the:per-map-into}}

\begin{proof}[Proof of Theorem \ref{the:per-map-into}]
We begin with the case $(a,0)$. Let $F_{(a,0)}:M_{(a,0)}\to N_{(a,0)}$ denote the map defined in Eq. \eqref{eq:F-def}. By Lemmas \ref{lem:jac-F}--\ref{lem:hyp-Hard-the}, $F_{(a,0)}$ satisfies the hypotheses of the Hadamard global diffeomorphism theorem. Consequently, $F_{(a,0)}$ is a diffeomorphism, and hence $\Theta_{(a,0)}^+:\Int \mathcal{A}_{(a,0)}^+\cap\{\eta>0\} \to \R^3$ is one-to-one, where we have used the fact that $\mathcal{A}_{(a,0)}^+$ is open. Lemma \ref{lem:per-map-nu-pos} them implies that $\Theta_{(a,0)}^{\pm}:\Int \mathcal{A}_{(a,0)}^{\pm} \to \R^3$ are one-to-one.  Lemma 35 then implies that \(\Theta^\pm_{(a,0)}\) are one-to-one. 

Applying the same argument to \(F_{(0,b)}\) proves the result for the family \((0,b)\).  
\end{proof}

    


\section{Conclusion}\label{sec:fut-work}
We have investigated the sub-Riemannian geodesics within the $2$-jet space of plane curves. Using rotational symmetry, Carnot translations, and dilations, we characterized the family of sub-Riemannian geodesics satisfying the asymptotic condition outlined in Conjecture \ref{conj}. These geodesics are determined by a momentum parameter \(\mu\in\Lambda\), where, modulo dilations, \(\Lambda\) is parametrized by $(a,b) \in \R^2$.

Theorem \ref{the:main} establishes that two families of geodesics with momentum $\mu \in \Lambda$ are metric lines, namely the cases $(a,0)$ and $(0,b)$. To prove Theorem \ref{the:main}, we constructed a sub-Riemannian manifold $\R^5_{(a,b)}$ and defined a period map that encodes the asymptotic behavior of these geodesics.

Theorem \ref{thm:main-2} extends the analysis to a general geodesic with momentum $\mu \in \Lambda$ and provides sufficient conditions for it to be a metric line. One key condition is that the period map is one-to-one. For the two families considered in Theorem \ref{the:main}, we establish this injectivity using the Hadamard global diffeomorphism theorem. Extending the analysis of the period map to the general case \(ab\neq0\) remains open.

\appendix

\section{The Jet Space as a Carnot Group}\label{ap:lie-alg}

The distribution $\D$ generates a Lie algebra of vector fields by iterating Lie brackets
$$ Y_\ell^i := [X,Y_\ell^{i-1}] =  \frac{\partial}{\partial u_\ell^{k-i}}, \;\; \ell=1,\dots,n,\;\;i=1,\ldots,k. $$
Here, we use the convention $Y_\ell^0:=Y_\ell$.

Note that the bracket relations above do not depend on the point  $j^k_{x_0}(f)$, so these vector fields are left-invariant and arise from the Lie algebra$\mathfrak{j}^k(\R,\R^n)$. Indeed, the restriction of these vector fields to the origin defines a stratified Lie algebra whose layers are
\begin{equation}
    \begin{split}
        V_1 & = \operatorname{span}\{X,Y_1,\ldots,Y_n \}, \\
        V_i & = \operatorname{span}\{Y_1^{i-1},\ldots,Y_n^{i-1} \}, \; i=2,\ldots,k+1.
    \end{split}
\end{equation}
Therefore, the Lie algebra is stratified as follows:
$$ \mathfrak{j}^k(\R,\R^n) = V_1 \oplus \cdots \oplus V_{k+1}. $$
For $\PJ$, the Lie algebra is 
$$\mathfrak{j}^2(\R,\R^2) = \mathrm{span}\{X,Y_1,Y_2,Y_1^1,Y_2^1,Y_1^2,Y_2^2\}.$$ 
The normal subgroup $\mathcal{N}$ introduced in Section \ref{subsec:two-jet} is defined by $\mathcal{N} = \exp(\mathfrak{n})$, where $\mathfrak{n}:= \mathrm{span}\{ Y_2,Y_2^1,Y_2^2\}$ is an ideal of the Lie algebra, as shown by the above Lie bracket relations. It follows that
$\mathfrak{j}^2(\R,\R^1) \simeq \mathfrak{j}^2(\R,\R^2)/\mathfrak{n}$, as mentioned earlier.

The Lie bracket relations show that $V_{i+1} = [V_1,V_i]$, for $i=1,\ldots,k$, and $[\mathfrak{j}^k(\R,\R^n),\mathfrak{j}^k(\R,\R^n)]$ is an abelian ideal of $\mathfrak{j}^k(\R,\R^n)$.  Consequently, $\kJ$ is a \ma Carnot group. 

A group $\G$ is \ma if $[\G,\G]$ is abelian \cite[Chapter 5]{robinson2012course}. Equivalently, $\G$ is \ma if and only if it contains a normal abelian subgroup $\Ag \triangleleft \G$ such that $\G/\Ag$ is abelian. Thus, choosing a maximal abelian normal subgroup $\Ag$ containing $[\G,\G]$, we may view a metabelian Carnot group $\G$ as an extension of $\G/\Ag$ by $\Ag$. This structure plays a pivotal role in the symplectic reduction of $T^*\G$ by $\Ag$. For further details on the construction of the reduced Hamiltonian $H_{\mu}$ and the symplectic reduction, see \cite{BravoDoddoli2024}.

In the case of $\kJ$, let $\Laa = \mathrm{span}\{Y_1,\ldots,Y_n,Y^1_1,\ldots,Y_n^k\}$. Then, $\Laa$ is a maximal abelian ideal, and $\Ag = \exp(\Laa)$ is a maximal abelian normal subgroup. Hence,  $\kJ/\Ag \simeq \R$.

\subsection{Symmetries}\label{subsub:sym}
The symmetries admitted by $\kJ$ are the following:

(\textbf{Carnot Translations}) By construction, the \sR structure is left-invariant under group translations. 

(\textbf{Dilations}) Carnot groups possess a natural family of dilations, forming a one-parameter group of automorphisms of $\G$. For every $\lambda \in \R\setminus\{0\}$, we denote the corresponding dilation by $\delta_\lambda$. If $\G$ is endowed with a left-invariant \sR metric, then dilations are compatible with the \sR metric in the sense that
$$ \mathrm{dist}_{\G}(\delta_{\lambda}g_1,\delta_{\lambda}g_2) = |\lambda| \mathrm{dist}_\G(g_1,g_2).$$
Moreover, if $\gamma(t)$ is a \sR geodesic parametrized by arc length, so is $\gamma_{\lambda}(t) = \delta_{\frac{1}{\lambda}}\gamma(\lambda t)$. By \cite[Lemma 2.5]{BravoDoddoli+2024}, it suffices to classify metric lines up to dilations and translations.

(\textbf{Rotations}) Consider the special orthogonal group $SO(n)$, and let $Q\in SO(n)$. We define an automorphism $\psi_Q:\mathfrak{j}^k(\R,\R^n) \to \mathfrak{j}^k(\R,\R^n)$ by specifying its values on the basis $\{X,Y_1,\ldots,Y_n,Y_1^1,\ldots,Y_n^k\}$:
    \begin{equation*}
        \begin{split}
        \psi_Q(X) & = X,\;\; \psi_Q(Y_j) = \sum_{\ell=1}^n Q_{j\ell}Y_\ell,\;\; j=1,\ldots,n, \\
        \psi_Q(Y_j^i) & = \sum_{\ell=1}^n Q_{j\ell}Y_\ell^i, \;\;j=1,\ldots,n,\;\;i=1,\ldots,k.  
        \end{split}
    \end{equation*}
    The automorphism $\psi_Q$ induces an isometry $\Psi_Q:\kJ \to \kJ$. By construction, $\psi_Q$ preserves all the Lie bracket relations; in particular, $\psi_Q(Y_j^i) = [\psi_Q(X),\psi_Q(Y_j^{i-1})]$. Since $Q$ is invertible, $\psi_Q$ defines a linear Lie automorphism. Since $\kJ$ is connected and simply connected,
     the Lie algebra automorphism $\psi_Q$ integrates to a Lie group automorphism $\Psi_Q:\kJ \to \kJ$ \cite[Theorem 5.6]{hall2015lie}. This automorphism preserves the distribution $\D$, and for every $g\in \kJ$ and $v \in \D_g$, $\|v\| = \| (\Psi_Q)_*v\|$. Hence, $\Psi_Q$ is a \sR isometry.

\section{Abnormal Geodesics}\label{Ap:abn}

\subsection{Proof of Proposition \ref{prp:abn-geo}}\label{ap:abn1}
\begin{proof}
    We compute the abnormal curves using Pontryagin’s maximum principle. Consider the following optimal control problem:
    \begin{equation}\label{eq:abn-geo}
        \dot{\gamma}(t) = v_x(t) X + v_1(t)Y_1+ \cdots + v_n(t)Y_n 
    \end{equation}
    subject to the boundary conditions $\gamma(0)$ and $\gamma(T)$. Note that Eq. \eqref{eq:abn-geo} implies $\dot{x} = v_x(t)$. By the Cauchy--Schwarz inequality, minimizing the sub-Riemannian length is equivalent, up to constant-speed parametrization, to minimizing the following action
\begin{equation}\label{eq:abn-geo-2}
\frac{1}{2}\int_0^T v_x^2+ v_1^2+ \ldots +v_{n}^2dt \to min.
\end{equation} 
For an abnormal extremal, the cost multiplier vanishes. Accordingly, the Hamiltonian associated with \eqref{eq:abn-geo} and \eqref{eq:abn-geo-2} is
    \begin{equation}\label{eq:ab-funct}
        H(P,j_{x}^k(f),v) = v_x P_X + v_1 P_{Y_1}+\ldots+v_nP_{Y_n},
    \end{equation}
    where $P_X, P_{Y_1},\ldots,P_{Y_n}$ denote the left-invariant momentum functions corresponding to the frame of left-invariant vector fields defining the distribution $\D$. Applying Pontryagin’s maximum principle yields a Hamiltonian system for the left-invariant momentum functions:
    \begin{align*}
        \dot{P}_x & = -v_1(t)P_{Y^1_1}-\cdots-v_n(t)P_{Y^1_n},\;\;&\\
        \dot{P}_{Y_\ell} & = v_x(t)P_{Y_\ell^1}, \;\;&\ell=1,\ldots,n,\\
        \dot{P}_{Y_\ell^{i}} & = -v_x(t)P_{Y^{i+1}_\ell},\;\;&\ell=1,\ldots,n,\;\;&i = 1,\ldots,k-1, \\
        \dot{P}_{Y_\ell^k} & = 0,\;\;&\ell=1,\ldots,n.
    \end{align*}
In addition, the maximum condition is given by
\begin{equation}\label{eq:pont-prin}
\max_{v \in \R^{n+1}} H(P(t),j_{x}^k(f)(t),v) = H(P(t),\widehat{\gamma}(t),\widehat{v}(t)), 
\end{equation}
where  $\widehat{\gamma}(t)$ is the optimal process, subject to the non-triviality condition $P(t) \neq 0$. 
Since the control set is \(\mathbb R^{n+1}\) and the Hamiltonian is linear in \(v\), the maximum condition can hold only if
\begin{equation}\label{eq:abn-proof-gama-5}
 P_{X} = 0,\;\;  P_{Y_\ell} = 0, \;\;\ell = 1,\dots,n. 
\end{equation}
Differentiating the identities in Eq. \eqref{eq:abn-proof-gama-5}, and using the Hamiltonian equations, we obtain
    \begin{equation}\label{eq:b-6}
    \begin{split}
        0 & = -\sum_{\ell=1}^n v_\ell(t)P_{Y^1_\ell},\;
        \;\text{and}\;\;0  = v_x(t)P_{Y_\ell^1}\;\;\ell=1,\ldots,n.
    \end{split}
    \end{equation}
    Suppose that \(v_x(t_0)\neq0\) for some \(t_0\). By continuity, \(v_x\neq0\) on a neighborhood of \(t_0\). It follows from Eq. \eqref{eq:b-6} that \(P_{Y_\ell^1}=0\) on this neighborhood for \(\ell=1,\ldots,n\). On the other hand, the Hamiltonian equations give $\dot{P}_{Y_\ell^1} = -v_xP_{Y_\ell^2}$. Since \(v_x\neq0\),  we obtain $P_{Y_\ell^2} = 0$. Repeating this argument yields \(P_{Y_\ell^i}=0\) for \(i=1,\ldots,k\) and \(\ell=1,\ldots,n\). In particular, \(P(t_0)=0\), contradicting the non-triviality condition.
    
    Therefore, \(v_x=0\), and hence $\dot{\gamma}(t) \in \mathrm{span} \{ Y_1,\ldots,Y_n\}$, meaning the $x$-component of the curve  remains constant. Since the vector fields \(Y_1,\ldots,Y_n\) commute, the singular curves lie in integral submanifolds tangent to \(\operatorname{span}\{Y_1,\ldots,Y_n\}\). Moreover, these vector fields form an orthonormal frame on each such submanifold, so the induced metric is Euclidean. Hence, its minimizing geodesics are straight lines.
\end{proof}

\subsection{Proof of Proposition \ref{prp:abn-mag-spa}}
\begin{proof}
       Adopting the approach outlined in Section \ref{ap:abn1}, consider the following optimal control problem
    \begin{equation}\label{eq:abn-proof-c}
    \dot{c}(t) = v_x(t \widetilde{X}+ v_1(t) \widetilde{Y}_1  + v_2(t) \widetilde{Y}_2, 
    \end{equation}
      with its corresponding action. The Hamiltonian associated with the abnormal case of this control problem is 
    \begin{equation}
        H(p_x,p_{\textbf{y}},p_{\textbf{z}},x,\textbf{y},\textbf{z},v) = v_x p_x + \sum_{i=1}^2 v_i(p_{y^i}+p_{z^i}P_i(x)).
    \end{equation}
Applying Pontryagin’s maximum principle yields a Hamiltonian system for the variables $p_{x}$, $p_{\textbf{y}}$, and $p_{\textbf{z}}$:
\begin{equation}\label{eq:abn-c}
\dot{p}_{x} = -\sum_{i=1}^2 v_i(t)p_{z^i}\frac{\partial P_i}{\partial x},\;\;\dot{p}_{y^i} = 0, \quad\dot{p}_{z^i} = 0,\quad i = 1,2.
\end{equation}
Thus, the maximum condition yields
\begin{equation}\label{eq:abn-proof-gama-6}
 p_{x} = 0,\;\;\text{and}\;\; 0 = p_{y^i} + p_{z^i}P_i(x),  \;\;i = 1,2. 
\end{equation}
Observe that the non-triviality condition requires $(p_{\textbf{y}},p_{\textbf{z}}) \neq (\mathbf{0},\mathbf{0})$. If $p_{\textbf{z}} = \mathbf{0}$, then $p_{\textbf{y}} = \mathbf{0}$, so we must have $p_{\textbf{z}} \neq \mathbf{0}$. Let  $p_{\textbf{z}} = (c_1,c_2) \in \R^2\setminus\{(0,0)\}$. Since \(p_z=(c_1,c_2)\neq(0,0)\), there exists an index \(i\in\{1,2\}\) such that \(c_i\neq0\). Differentiating the corresponding equation in \eqref{eq:abn-proof-gama-6} and using $c_i\neq 0$, we obtain 
$$ P_i'(x)v_x=0. $$
We claim that \(v_x=0\). Indeed, suppose that \(v_x(t_0)\neq0\) for some \(t_0\). By continuity, \(v_x\neq0\) on a neighborhood of \(t_0\), and hence \(P_i'(x)=0\) on this neighborhood. Since \(P_i\) is a nonconstant polynomial, \(P_i'\) has only finitely many zeros. Therefore, \(x\) is constant on this neighborhood, which contradicts \(\dot x=v_x\neq0\). Thus, \(v_x=0\), and consequently \(x\) is constant.

 For the curve to be nonconstant, $(v_1(t),v_2(t))$ cannot vanish identically. By combining Eqs. \eqref{eq:abn-c} and \eqref{eq:abn-proof-gama-6}, we obtain
$$ 0 = (v_1(t),v_2(t)) \cdot (c_{1}\frac{\partial P_1}{\partial x},c_{2}\frac{\partial P_2}{\partial x}). $$
Consequently, $c(t)$ lies in an integral plane whose tangent space is spanned by $\{ \tilde{Y}_1,\tilde{Y}_2\}$. Thus, the restriction of the projection $pr:\R^5_{(a,b)} \to \R^{3}$ to this integral submanifold is an isometry. Therefore,  all the abnormal geodesics are line geodesics. Since line geodesics are normal geodesics with a constant \(\mathbb R^n\)-valued polynomial, we conclude that there are no strictly abnormal geodesics
in $\R^5_{(a,b)}$.
\end{proof}

\subsection*{Acknowledgments.}
The authors thank the Department of Mathematics at the University of Michigan for organizing the 2025 REU project, during which this research was conducted. In particular, we thank Prof. Asaf Cohen for his time and dedication in making the project possible. We also thank Annie Winkler for her support and help before and during the project. Additionally, Daniella Catalá gratefully acknowledges the support of Prof. Anthony Bloch through his NSF grant (\href{https://www.nsf.gov/awardsearch/showAward?AWD_ID=2103026&HistoricalAwards=false}{DMS-2103026}), and Miriam Vollmayr-Lee gratefully acknowledges the support of Prof. Ralf Spatzier through his NSF grant (\href{https://www.nsf.gov/awardsearch/showAward?AWD_ID=2404309&HistoricalAwards=false}{DMS-2404309}).

\bibliographystyle{plain}
\bibliography{bibli}

\begin{thebibliography}{10}

\bibitem{agrachev2019comprehensive}
A.~Agrachev, D.~Barilari, and U.~Boscain.
\newblock {\em A Comprehensive Introduction to sub-Riemannian Geometry}, volume 181.
\newblock Cambridge University Press, 2019.

\bibitem{BravoDoddoli+2024}
A.~Bravo-Doddoli.
\newblock Metric {Lines} in the {Jet Space}.
\newblock {\em Analysis and Geometry in Metric Spaces}, 12(1):20240016, 2024.

\bibitem{bravododdoligeltype}
A.~Bravo-Doddoli.
\newblock Metric {Lines} in {Engel-type} {Groups}, 2025.

\bibitem{BravoDoddoli2024}
A.~Bravo-Doddoli, E.~Le~Donne, and N.~Paddeu.
\newblock Symplectic {Reduction} of the {sub-Riemannian Geodesic Flow for Metabelian Nilpotent Groups}.
\newblock {\em Geometric Mechanics}, 01(01), March 2024.

\bibitem{bravo2022geodesics}
A.~Bravo-Doddoli and R.~Montgomery.
\newblock Geodesics in {Jet} {Space}.
\newblock {\em Regular and Chaotic Dynamics}, 27(2):151--182, 2022.

\bibitem{hakavuori2023blowups}
E.~Hakavuori and E.~Le~Donne.
\newblock Blowups and {Blowdowns} of {Geodesics} in {Carnot} {Groups}.
\newblock {\em Journal of Differential Geometry}, 123(2):267--310, 2023.

\bibitem{hall2015lie}
B.~Hall.
\newblock {\em Lie Groups, Lie Algebras, and Representations: An Elementary Introduction}.
\newblock Graduate Texts in Mathematics. Springer International Publishing, 2015.

\bibitem{krantz2002implicit}
S.~G. Krantz and H.~R. Parks.
\newblock {\em The Implicit Function Theorem: History, Theory, and Applications}.
\newblock Springer Science \& Business Media, 2002.

\bibitem{le2025metric}
E.~Le~Donne.
\newblock {\em Metric {Lie} groups. {Carnot}-{Carath{\'e}odory} {Spaces from} the {Homogeneous Viewpoint}}, volume 306 of {\em Grad. Texts Math.}
\newblock Cham: Springer, 2025.

\bibitem{montgomery2002tour}
R.~Montgomery.
\newblock {\em A {Tour} of {subRiemannian} {Geometries}, {Their} {Geodesics} and {Applications}}.
\newblock Number~91. American Mathematical Soc., 2002.

\bibitem{robinson2012course}
D.J.S. Robinson.
\newblock {\em A Course in the Theory of Groups}, volume~80.
\newblock Springer Science \& Business Media, 2012.

\bibitem{warhurstjet}
B.~Warhurst.
\newblock Jet {Spaces} as {Nonrigid} {Carnot} {Groups}.
\newblock {\em Journal of Lie theory}, 15(1):341--356, 2005.

\end{thebibliography}

\end{document}